\documentclass[a4paper,12pt]{amsart}

\usepackage{amssymb}
\usepackage{amsmath} 
\usepackage{amsthm}
\usepackage[cp 1250]{inputenc} 
\usepackage[dvips]{graphicx}
\usepackage[T1]{fontenc}

\usepackage{color}

\newtheorem{thm}{\sc Theorem}
\newtheorem{lemma}[thm]{\sc Lemma}

\newtheorem{remark}[thm]{\sc Remark}
\newtheorem{cor}[thm]{\sc Corollary}
\newtheorem{prop}[thm]{\sc Proposition}

\theoremstyle{definition} \newtheorem{defi}[thm]{\sc Definition}
\newtheorem{proc}[thm]{\sc Procedure}
\newtheorem{ob}[thm]{\sc Observation}

\numberwithin{thm}{section}
\newcommand{\one}{{\bf 1}} 
\newcommand{\U}{\mathbf{1}}
\newcommand{\vk}{\varkappa}
\newcommand{\vp}{\varphi}

\newcommand{\A}{{\mathcal A}}
\newcommand{\cG}{{\mathcal G}}
\newcommand{\cH}{{\mathcal H}}
\newcommand{\cR}{{\mathcal R}} 
\newcommand{\C}{{\mathbb C}} 
\newcommand{\cP}{{\mathcal P}} 
\newcommand{\e}{\varepsilon} 
\newcommand{\hl}[1]{\textbf{\textit{#1}}} 
\DeclareMathOperator{\id}{\normalfont\mathsf{id}} 

\newcommand{\N}{{\mathbb N}} 
\newcommand{\pol}[1][N]{SU_q(#1)} 
\newcommand{\R}{{\mathbb R}} 

\newcommand{\lf}{\vspace{1ex}}
\newcommand{\ls}{\text{{\small $\textsf{span}$}}}
\newcommand{\cls}{\overline{\ls}}

\usepackage{hyperref}

\numberwithin{equation}{section}

\title{{~}\\\vspace{-10ex}
Hunt's Formula for $SU_q(N)$ and $U_q(N)$}
\author{Uwe Franz, Anna Kula, J.\ Martin Lindsay, Michael Skeide}

\begin{document}

\begin{abstract}
We provide a Hunt type formula for the infinitesimal generators of L\'evy 
process on the quantum groups $SU_q(N)$ and $U_q(N)$. In particular, we 
obtain a decomposition of such generators into a gaussian part and a `jump 
type' part determined by a linear functional
{that resembles}
the functional induced by the L\'evy measure. The jump part on $SU_q(N)$ decomposes further into parts that live on the quantum subgroups  $SU_q(n)$, $n\le N$. Like in the classical Hunt formula for locally compact Lie groups, the ingredients become unique once a certain projection is chosen. There are analogous result  for $U_q(N)$.
\end{abstract}

\maketitle

{
\addtolength{\parskip}{-3mm}

\vspace{-3ex}
\tableofcontents

}

\newpage

\section{Introduction}
\label{sec_intro}

Let $\mathcal G$ denote a locally compact Lie group. All information about a L\'evy process with values in $\mathcal G$ can be captured (up to stochastic equivalence) by its infinitesimal generator, a densely 
defined linear functional $\psi$ on the $C^*$-algebra $C_0(\cG)$ of bounded continuous 
functions on $\mathcal G$ vanishing at infinity. The domain may be thought of as those functions that 
{possess}
a second order Taylor expansion around the neutral element of $\mathcal 
G$, $e$. Hunt's formula \cite{Hun56}, a generalization of the L\'evy-Khintchine 
formula for $\R$ (see for instance Applebaum \cite{app04} or Sato \cite{sato99}), asserts that
$$
\psi(f)
~=~
\psi_G(f)+\int_{{\mathcal G}\backslash\{e\}}[\cP(f)](g)\,d{\mathcal L}(g).
$$
The \hl{gaussian part} $\psi_G$ is a linear combination of first and 
second order derivatives at the neutral element, $\mathcal P$ is an arbitrary hermitian 
projection that takes away the linear terms, and $\mathcal L$ is the \hl{L\'evy measure} (which may have a singularity up to order two at $e$). Defining the \hl{L\'evy functional} $L(f):=\int_{{\mathcal G}\backslash\{e\}}f(g)\,d{\mathcal L}(g)$ on the functions which together with their first derivatives vanish at $e$ and putting $\psi_L=L\circ\cP$, this reads
\begin{equation} \label{eq_LK_decomp}
\psi
~=~
\psi_G+\psi_L.
\end{equation}
Since the integral may be viewed as a mixture of point evaluations, and since, for fixed $g\ne e$, a generating functional of the form $f\mapsto f(g)-f(e)$ generate a {\it jump process}, $L\circ\cP$ is sometimes also referred to as the \hl{jump part}. (In the case $\cG=\R$, we get a {\it compound Poisson process}).

\begin{remark}
Note that for the pure point evaluation, no $\cP$ is necessary; $\cP$ is necessary only to deal with the singularity at $e$ of the L{\'e}vy measure. Note, too, that there is (usually) no canonical choice for $\cP$; this is, why in literature the classical {\it L{\'e}vy-Khintchine formula} for $\cG=\R$ may look quite different depending on the reference. The projection $\cP$ will keep us quite busy; see Subsections \ref{sec_S_trip}--\ref{ssec_gaussian} and  Sections \ref{sec_gaussian} and \ref{sec_uqn}.
\end{remark}

If $\mathcal G$ is compact, the well known Tannaka-Krein duality (see, for instance, Hewitt and Ross \cite[Section VII.30]{hewitt_ross70}) asserts that the \hl{coefficient algebra} ${\mathcal R}({\mathcal G})$ (consisting of coefficients of finite-dimensional representations of $\mathcal G$) is a norm dense $*$-subalgebra of the $C^*$-algebra $C({\mathcal G})$ of continuous functions. Actually, ${\mathcal R}({\mathcal G})$ is a commutative \hl{Hopf $*$-algebra}, and the structure of the topological group $\mathcal G$ may be recovered from the Hopf $*$-algebra ${\mathcal R}({\mathcal G})$. 

More generally, a \hl{compact quantum group} $\cG=(C({\cG}),\Delta)$ in the sense of Woronowicz \cite{woronowicz98}, 
which is roughly speaking a unital $C^*$-algebra $C({\cG})$ with an 
additional structure reflecting the group properties on the level of functions 
on a group, always contains a dense $*$-subalgebra ${\mathcal R}({\cG})$ that may be turned into a Hopf $*$-algebra (\cite[Theorem 
1.2]{woronowicz98}). This opens up the way to apply Sch\"{u}rmann's theory of \hl{quantum L\'evy processes on $*$-bialgebras} \cite{schurmann93} to both situations.

Like their classical counterparts, L\'evy processes on $*$-bialgebras are classified (up to quantum stochastic equivalence) by their \hl{generating 
functionals}. A generating functional is a linear functional on $\cR(\cG)$ fulfilling certain algebraic conditions; see Subsection \ref{sec_gen_func}. Finding generating functionals amounts to the solution of a cohomological problem; see Subsection \ref{sec_S_trip}. Decomposing a generating functional $\psi$ into a sum, means, roughly speaking, that the cohomological problem has to be solved for the constituents, individually; see Subsection \ref{sec_G_D_A}. Of course, to merit being called a \hl{L{\'e}vy-Khintchine decomposition}, a decomposition as in \eqref{eq_LK_decomp} has to satisfy more: We have to say what a {\it gaussian part} is (a {\it quadratic} part, in the sense that it is $0$ on all monomials of degree $3$ and higher) and what a {\it completely non-gaussian part} is (no gaussian part can be subtracted); see Subsection \ref{ssec_gaussian}.

A ``true'' \hl{Hunt formula} goes further. It includes an explicit description of the gaussian generating functionals. And it includes a certain approximation property that justifies to call the completely non-gaussian part a jump part. The approximation property we require to call an approximated generating functional a {\it jump part}, is explained in Subsection \ref{sec_G_D_A}; a justification/motivation is given in Remark \ref{JPrem}.

It is noteworthy that not all quantum groups allow to decompose every generating functional into a (maximal) gaussian and a completely non-gaussian part; see Franz, Gerhold, and Thom \cite[Proposition 4.3]{franz_gerhold_thom}; therefore, already the answer to the question if the decomposition problem has a solution or not, depends on the example under consideration.

Sch{\"u}rmann and Skeide \cite{Ske94,schurmann+skeide98} established Hunt's formula in this sense on Woronowicz's $SU_q(2)$ \cite{Wor87b}. Skeide 
\cite{skeide99} applied Sch\"{u}rmann's ideas to obtain a quick proof of Hunt's formula for compact Lie groups. In these notes, we deal with the case $SU_q(N)$. We do not only find Hunt's formula for $SU_q(N)$.  In Theorem \ref{mainthm} we find that a generating functional $\psi$ decomposes (in a sense, uniquely) as
$$
\psi
~=~
\psi_G+L_2\circ\cP+\ldots+L_N\circ\cP,
$$
where again $\psi_G$ is a gaussian part, and where $L_n$ ($2\le n\le N$) are (extensions to $SU_q(N)$ of) L\'evy functionals on $SU_q(n)\subset SU_q(N)$. 
{\it En passant}, in Section \ref{sec_uqn}, we derive similar results for the quantum group $U_q(N)$.

The techniques are inspired quite a bit by \cite{Ske94,schurmann+skeide98} for the decomposition and by \cite{skeide99} for the gaussian part. But the case of general $N$ is more involved.  It turns out that some results on $SU_q(2)$ fail for $N\ge3$. For instance, for $N\geq 3$ in the gaussian case the cohomological problem may not always be solved; see Corollary \ref{ngccor}. Also, for $N=2$ the L{\'e}vy functionals $L$ without gaussian part can be parametrized by {\bf all} vectors in a certain representation Hilbert space, whereas for $N\geq 3$ this is no longer true; see Proposition \ref{prop_counterexample}.

\medskip
The paper is organized as follows. Section \ref{sec_prelim} presents preliminary results; some of them are new. (For instance, the treatment of the projection $\cP$, in particular, in connection with subgroups, is new. Also, for the analytic key lemma in Subsection \ref{keySSEC}, though probably {\it folklore}, we did not find a source; this lemma also drastically simplifies the case $SU_q(2)$, and is responsible for that we need not reference to \cite{Ske94,schurmann+skeide98} for more than motivation.) Section \ref{sec_gaussian} deals with the gaussian case and our choice of $\cP$. Section \ref{decompSEC} presents the actual decomposition for $SU_q(N)$, while Section \ref{sec_class} pushes forward to $SU_q(N)$ the parametrization results from Skeide \cite[Section 4.5]{Ske94} or \cite[Section 4.3]{skeide99}. Section \ref{sec_uqn} deals with $U_q(N)$. In the final section we discuss some open problems for future work.

\medskip
\noindent
{\bf Conventions and choices.~}

A (\hl{topological}) \hl{compact quantum group} $\mathcal G=(C(\cG),\Delta)$ (Woronowicz \cite{woronowicz98}) is a unital $C^*$-algebra $C(\cG)$ 
with a unital $*$-ho\-mo\-mor\-phism $\Delta\colon C(\cG)\rightarrow 
C(\cG)\otimes_{min} C(\cG)$ that is 
{\it coassociative} ($(\Delta\otimes\id)\circ\Delta=(\id\otimes\Delta)\circ\Delta$)
and that satisfies the {\it quantum cancellation rules} $\cls(C(\cG)\otimes\U)\Delta(C(\cG))=
{C(\cG)\otimes_{\rm min} C(\cG)}
=\cls\Delta(C(\cG))(\U\otimes C(\cG))$.

An (\hl{algebraic}) \hl{compact quantum group} or \hl{CQG-algebra} (Dijkhuizen and Koornwinder \cite{DiKo94}) is a Hopf $*$-algebra $(\cG,\Delta,\e)$ (see Subsection \ref{sec_gen_func}) that is spanned by the coefficients of its finite-dimensional unitary corepresentations. (Equivalently, a CQG-algebra is  a Hopf $*$-algebra with a {\it Haar state}. But the original definition suits our situation better.)

Algebraic and topological compact quantum groups are two sides of the same coin, compact quantum groups. (See the books by Klimyk and Schm{\"u}dgen \cite[Section 11.3]{KlSchm97} or by Timmermann \cite[Section 5.4]{timmermann08}.) Especially, a topological compact quantum group $(C(\cG),\Delta)$ contains a (unique) dense $*$-subalgebra $\cR(\cG)$ (the linear hull of the coefficients of its finite-dimensional corepresentations) such that the restriction of $\Delta$ to $\cR(\cG)$ maps into the algebraic tensor product $\cR(\cG)\otimes\cR(\cG)$; existence of counit and antipode are theorems.%
\footnote{
Frequently, in the literature $\cR(\cG)$ is also denoted
{$\mathsf{Pol}(\cG)$}.
}

Sch{\"u}rmann's theory of {\it quantum L{\'e}vy processes}, from which we take the notion of generating functionals, is about (\hl{algebraic}) \hl{quantum semigroups} (or $*$-bialgebras) and (\hl{algebraic}) \hl{quantum groups} (or Hopf $*$-algebras). Even the notion quantum subgroups of a (topological) compact quantum group $\cG$ is referring to $\cR(\cG)$ rather than to $C(\cG)$. Therefore:

In these notes we view compact quantum groups (like $SU_q(N)$ and $U_q(N)$) exclusively as CQG-algebras. (Our exposition of  $SU_q(N)$ and $U_q(N)$ in Subsection \ref{ssec_suqn} follows Koelinks exposition of $U_q(N)$ in \cite{koelink91}) {\bf We shall write $\cG$ to mean $\cR(\cG)$}; the $C^*$-algebra $C(\cG)$ does not occur.

In Sch{\"u}rmann's theory, $*$-representations of $\cG$ are by (possibly unbounded) operators on pre-Hilbert spaces -- and (thinking, for instance, about the quantum groups constructed from Lie algebras) this is good so. But representations of CQG-algebras are all by bounded operators. After we introduced $SU_q(N)$ and $U_q(N)$ in Subsection \ref{ssec_suqn}, we will consequently complete the occurring pre-Hilbert spaces. But before that, the discussion (some of it new) is quite general; it would be a pity to write it down in a way that is not directly quotable from future papers just because we completed too early. (See also Footnote \ref{FN} in Subsection \ref{ssec_gaussian}.)

Last but surely not least, we emphasize already now that the property of a linear functional on $\cG$ to be a generating functional, makes reference only to the $*$-algebra structure of $\cG$ and to the counit $\e$; no comultiplication is needed and no antipode is needed. Throughout these notes, with one exception (in the definition of quantum subgroup in Subsection \ref{sec_Stqsub}), we ignore the antipode. The comultiplication, though not strictly necessary, comes in useful in a couple of places. (See the last sentence in Subsection \ref{sec_S_trip} and the proof of Proposition \ref{prop_counterexample}.) Therefore, we carry it along.

\newpage

\addtocontents{toc}{\vspace{1.5ex}}
\section{Preliminaries} \label{sec_prelim}

\subsection{Generating functionals of L\'evy processes} \label{sec_gen_func}

Classical L{\'e}vy processes, in the most general formulations, take values in a group or even only in a semigroup. The latter allows to say what the increments of the process are, and to define the most important property a L{\'e}vy process has to satisfy: Namely, to have {\it independent increments}. The passage from the classical world to the noncommutative world is frequently made by {\it dualization}: Replace spaces (for instance, probability spaces or semigroups) by $*$-algebras of complex functions on these spaces, look how all the structures of the spaces are reflected by additional structures of these function algebras and take these as axioms, but in the end forget that the algebras are commutative. Probability spaces become $*$-algebras with a state, semigroups become $*$-bialgebras, and semigroup-valued random variables on a probability space become, under dualization, $*$-algebra homomorphisms from the $*$-bialgebra into the quantum probability space. A {\it quantum L{\'e}vy process} is, therefore, a family of such homomorphisms fulfilling certain extra properties.

Fortunately, quantum L{\'e}vy processes are determined by their so-called {\it generating functionals} and, fortunately, the only scope of these notes is to examine the structure of such generating functionals on the quantum groups $SU_q(N)$ and $U_q(N)$. 
We do not really need to know what a quantum L{\'e}vy process is, but only the properties that make a functional a generating functional. For details about quantum L\'evy processes we refer the reader to 
Sch{\"u}rmann \cite{schurmann93}, Meyer \cite[Chapter VII]{meyer93}, Franz \cite{franz06} or the recent 
survey \cite{franz16}. An abstract reconstruction of a L{\'e}vy process from a generating functional can be found in Sch{\"u}rmann \cite{schurmann93} (basically, but in a more general context, \cite[Proposition 1.9.5]{schurmann93} and the discussion preceding it and summarized as a theorem in \cite[Corollary 1.9.7]{schurmann93}), or (reducing it to Bhat and Skeide \cite{BhSk00}) in Skeide \cite{Ske05b}. A proof that the reconstruction can be done on a Fock space, can be found in Sch{\"u}rmann \cite[Theorem 2.3.5]{schurmann93} (using quantum stochastic calculus) and in Sch{\"u}rmann, Skeide, and Volkwardt \cite{ScSV10} (using techniques as for Trotter products and Arveson systems \cite{Arv89}).

Recall from the conventions that we opted to set up compact quantum groups the algebraic way (as CQG-algebras), as these meet better Sch{\"u}rmann's setting. Recall, too, that we opted to keep these preliminaries as applicable as possible also for quantum (semi)groups that are not necessarily (locally) compact. Therefore, in the first few subsections (until \ref{ssec_gaussian}), representations are by (not necessarily bounded) operators on pre-Hilbert spaces. Only after we
{introduce}
$U_q(N)$ and $SU_q(N)$ in Subsection \ref{ssec_suqn}, we always assume pre-Hilbert spaces completed.

\medskip
An (algebraic) \hl{quantum semigroup} $\cG$ is a \hl{$*-$bialgebra}, that is, $\cG$ is a complex 
involutive unital algebra $\cG$ with a unital $*$-homomorphism 
$\Delta\colon\cG\rightarrow\cG\otimes\cG$ (the \hl{comultiplication}) from $\cG$ 
into the algebraic tensor product $\cG\otimes\cG$ and a unital $*$-homomorphism 
$\e\colon\cG\rightarrow\C$ (the \hl{counit}) from $\cG$ into the complex numbers fulfilling \hl{coassociativity}
$$ (\Delta\otimes\id_\cG) \circ \Delta = (\id_\cG\otimes\Delta) \circ \Delta$$
and the \hl{counit property}
$$ (\e\otimes\id_\cG)\circ \Delta = \id_\cG = (\id_\cG\otimes\e) \circ \Delta.$$
An (algebraic) quantum {\bf group} would be a Hopf $*$-algebra, that is, a $*$-bialgebra with an additional structure, the so-called {\it antipode}. As we do not need the antipode (its definition would require to introduce the multiplication map on $\cG\otimes\cG$ or the Sweedler notation), we do not address it in this paper.

The comultiplication, instead, together with the induced notion of convolution we discuss in a second, is, though not strictly necessary for our results, frequently useful. In any case, for the basic understanding why generating functionals occur as infinitesimal generators of convolution semigroups, the comultiplication is indispensable.

Using the comultiplication, we define the 
\hl{convolution} of two linear functionals $\varphi$ and $\psi$ on $\cG$ as
$$
\varphi\star\psi
~:=~
(\varphi\otimes\psi)\circ\Delta.
$$
By coassociativity, the product $\star$ turns the dual $\cG'={\mathcal L}(\cG,\C)$ into an (associative) algebra, and by the counit property the counit 
$\e$ is a unit for $\cG'$.

The {\it fundamental theorem of coalgebras} asserts that every element of a coalgebra is contained in a finite-dimensional subcoalgebra; see, for instance, Abe \cite[Corollary 2.2.14(i)]{Abe80}. This can be used easily to show that there is a complex functional calculus for entire functions, including (pointwise) differentiation and integration with respect to parameters. Therefore, there is a \hl{convolution exponential} $e_\star^\psi:=\sum_{n=0}^\infty\frac{\psi^{\star n}}{n!}$ and the formula $\varphi_t=e_\star^{t\psi}$ establishes a one-to-one correspondence between (pointwise) continuous convolution semigroups $(\varphi_t)_{t\ge0}$ and their infinitesimal generators $\psi=\frac{d}{dt}\big|_{t=0}\varphi_t$.

Like a classical L{\'e}vy process is determined by a convolution semigroup of probability measures, a quantum L{\'e}vy process is determined by a convolution semigroup of states. So, we are led to the question, when does a convolution semigroup $\varphi_t=e_\star^{t\psi}$ 
consist of \hl{states}, that is, of \hl{positive} ($\varphi_t(b^*b)\ge0$) 
\hl{normalized} ($\varphi_t(\U)=1$) linear functionals? By looking at $\frac{d}{dt}\big|_{t=0}\varphi_t$,   we easily check that if all $\varphi_t$ are states, then $\psi$ satisfies the following conditions:
\begin{itemize}
\item
$\psi$ is \hl{hermitian}, that is, $\psi(b^*)=\overline{\psi(b)}$.
\item
$\psi$ is \hl{$0$-normalized}, that is, $\psi(\U)=0$.
\item
$\psi$ is \hl{conditionally positive}, that is, $\psi(b^*b)\ge0$ whenever 
$b\in\ker\e$.
\end{itemize}
The best way to show that these conditions are also sufficient for that the convolution semigroup generated by $\psi$ consists of states, is by 
reconstructing from $\psi$ a L\'evy process that has this semigroup as convolution semigroup, see Sch{\"u}rmann \cite[Theorem 2.3.5]{schurmann93} or Sch{\"u}rmann, Skeide, and Volkwardt \cite{ScSV10}. We, therefore, say:

\begin{defi} \label{genfundef}
A \hl{generating functional} (for a quantum L{\'e}vy process) on a $*$-bialgebra $\cG$ is a linear functional 
$\psi\colon\cG\rightarrow\C$ that is hermitian, $0$-normalized, and conditionally 
positive.
\end{defi}

We see, to check if a linear functional on $\cG$ is a generating functional,  we only refer to the $*$-algebra structure of $\cG$ (in terms of positivity and being hermitian) and to the counit $\e$ (in terms of its kernel $\ker\e$); there is no reference to the comultiplication. The kernel of $\e$, on the other hand, and related structures are so important that we introduce already now the corresponding notation we are going to use throughout. We define for all $n\ge1$ 
$$K_n:=\ls(\ker\e)^n,$$ 
the span of all products of $n$ elements in $\ker\e$. (In particular, $K_1=\ker\e$.) We put $K_0:=\cG$. We also put 
$K_\infty:=\bigcap_{n\ge1}K_n$. Obviously, all $K_n$ ($0\le n\le\infty$) are $*$-ideals and $K_n\supset K_{n+1}\supset K_\infty$.

\medskip
\subsection{Sch{\"u}rmann triples and the projection $\cP$} \label{sec_S_trip}

Like with other infinitesimal generators of positive-type semigroups (generators of semigroups of positive definite kernels, semigroups of completely positive maps), there is a sort of GNS-construction also for generating functionals of quantum L{\'e}vy processes, their so-called {\it Sch{\"u}rmann triple}. Since every generating functional has a Sch{\"u}rmann triple, finding all Sch{\"u}rmann triples will provide us with all generating functionals; this will be our strategy.

Following Sch\"{u}rmann \cite[Section 2.3]{schurmann93}, for a conditionally positive functional $\psi$ on $\cG$ we define a positive sesquilinear form on $K_1$ by $(a,b)\mapsto\psi(a^*b)$, we divide out the null-space ${\mathcal N}_\psi:=\{b\in K_1\colon\psi(b^*b)=0\}$, and we form the quotient space $D_\psi:=K_1/{\mathcal N}_\psi$ which inherits a pre-Hilbert space structure by defining the inner product $\langle a+{\mathcal N}_\psi,b+{\mathcal N}_\psi\rangle:=\psi(a^*b)$. 

Let $\eta_\psi\colon\cG\rightarrow D_\psi$ denote the quotient map $ b\mapsto b+{\mathcal N}_\psi$ ($b\in K_1$) extended by $\eta_\psi(\U):=0$ to all of $\cG$. This makes sense, because of the following crucial fact:

\begin{ob} \label{ideob}
Every $b\in\cG$ can be written uniquely as 
$k+\U\lambda$ with $k\in K_1$. Indeed, necessarily $\lambda=\e(b)$. So, if we define the canonical projection onto $K_1$ as $\id-\U\e\colon b\mapsto b-\U\e(b)$, then, necessarily, $k=(\id-\U\e)(b)$. (Similar considerations are important in the discussion around Lemma \ref{PElem}.)
\end{ob}

Of course, $\eta_\psi(\cG)=\eta_\psi(K_1)=D_\psi$. As usual, an application of Cauchy-Schwarz inequality shows that for each $a\in\cG$ we 
(well)define%
\footnote{
Of course, the prescription does determine a map $\pi(a)$ -- provided such map exists; the question with ``defining'' by such determining
{prescriptions},
is whether the map actually {\bf does} exists. The term ``\hl{we (well)define}'' (and analogous variants) is shorthand for ``we {\bf attempt} to define a map by the following prescription, and it turns out that such map {\bf is} well-defined''.
}
a map $\pi_\psi(a)\colon\eta_\psi(b)\mapsto\eta_\psi(ab)$ $(b\in K_1$) on $D_\psi$. 
One easily verifies that $b\mapsto\pi_\psi(b)$ defines a unital $*$-representation $\pi_\psi\colon\cG\rightarrow{\mathcal L}^a(D)$ (the set of 
all adjointable maps $D\rightarrow D$). Using again that $b=(\id-\U\e)(b)+\U\e(b)$ for all $b\in\cG$, we see that 
$\eta_\psi$ is a \hl{$\pi_\psi$-$\e$-cocycle}, that is,
\begin{equation}\label{*}
\eta_\psi(ab) ~=~ \pi_\psi(a)\eta_\psi(b)+\eta_\psi(a)\e(b)
\end{equation}
for all $a,b\in\cG$. Taking into account that $\e$ is fixed, we frequently say $\eta$ is a cocycle \hl{with respect to $\pi$}. By construction, $\eta_\psi$ fulfills 
$\langle\eta_\psi(a),\eta_\psi(b)\rangle~=~\psi(a^*b)$ for  $a,b\in K_1$. Since $\psi(\U)=0$, this is the same as
\begin{equation}\label{**}
\langle\eta_\psi(a),
\eta_\psi(b)\rangle~=~\psi(a^*b)-\psi(a^*)\e(b)-\e(a^*)\psi(b)
\end{equation}
for all $a,b\in\cG$. So, the $\e$-$\e$-2-coboundary of $\psi$ is the map
{$(a,b)\mapsto-\langle\eta_\psi(a^*),
\eta_\psi(b)\rangle$}.
We say the ($\e$-$\e$-)2-coboundary of $\psi$ is \hl{$\eta_\psi$-induced}. Therefore:

\lf
\begin{defi} \label{def_triple}
A \hl{Sch\"{u}rmann triple} is a triple $(\pi,\eta,\psi)$ 
consisting of 
\begin{itemize}
 \item 
a unital $*$-representation $\pi\colon\cG \to {\mathcal L}^a(D)$ 
on some pre-Hilbert space $D$,  

\item
a $\pi$-$\e$-cocycle $\eta\colon \cG \to D$, and 

\item 
a linear functional $\psi\colon\cG\to\C$ whose $\e$-$\e$-$2$-coboundary is $\eta$-induced. 
\end{itemize}
We say, $(\pi,\eta,\psi)$ is \hl{cyclic} if $\eta$ is \hl{cyclic}, that is, if $\eta(\cG)=D$.
\end{defi}

Every generating functional is part of the Sch{\"u}rmann triple $(\pi_\psi,\eta_\psi,\psi)$, the output of the GNS-construction preceding Definition \ref{def_triple}. We say, $(\pi_\psi,\eta_\psi,\psi)$ is \hl{the} Sch\"{u}rmann triple 
associated with the generating functional $\psi$. Note that the Sch{\"u}rmann triple of $\psi$ is cyclic by construction. Moreover, if  $\psi$ is part of any other  Sch{\"u}rmann triple, say, $(\pi,\eta,\psi)$, then $\eta_\psi(b)\mapsto\eta(b)$ defines an isometry $v\colon D_\psi\rightarrow D$, which intertwines the representation in the sense that $\pi(a)v=v\pi_\psi(a)$ for all $a\in\cG$. If also $(\pi,\eta,\psi)$ is cyclic, then $v$ is even unitary and $\pi= v\pi_\psi v^*$. (Recall that a unitary between pre-Hilbert spaces has an adjoint, namely, its inverse.) In this sense, cyclic Sch{\"u}rmann triples are determined by $\psi$ up to unitary equivalence.

It is noteworthy that the construction of a Sch{\"u}rmann triple for a generating functional $\psi$ went $\psi$ $\leadsto$ ($D_\psi$ and) $\eta_\psi$ $\leadsto$ $\pi_\psi$. For finding all generating functionals, we rather proceed the opposite way:

\lf
\begin{proc}{~} \label{proc}

\begin{enumerate}
\item
Find all $*$-representations $\pi$;

\item
find all cocycles $\eta$ with respect to $\pi$;

\item
find all linear functionals $\psi$ with $\eta$-induced $2$-coboundaries;

\item
exclude all those that are not generating functionals.
\end{enumerate}
\end{proc}

\noindent
Since every generating functional has a Sch{\"u}rmann triple, in that way we surely will find all generating functionals. (And even if, for some quantum semigroup, we should not succeed in full generality for some of the steps, the procedure still promises to be a rich source for generating functionals; this way was quite successful for quantum L{\'e}vy processes on the Lie algebra $s\ell(2)$ in Accardi, Franz, and Skeide \cite{AFS02}.)

\begin{remark} \label{cycrem}
In Definition \ref{def_triple}, the condition that $\pi$ is unital (so that $\pi(\U)=\id_D$) is just for convenience. Indeed, if $\pi$ is not unital, then $\pi(\U)$ is still a projection, and \eqref{*} shows $\pi(\U)\eta(\cG)=\eta(\cG)$, so we may restrict to $\pi(\U)D$. By the same computation, we see that we may actually restrict to the invariant subspace $\eta(\cG)$, making the triple cyclic. However, while nonunital $\pi$ is usually just annoying, for the purpose to proceed as $\pi$ $\leadsto$ $\eta$ $\leadsto$ $\psi$ it is very convenient, for formal reasons (see the discussion following Proposition \ref{Stsumprop} and the beginning of Section \ref{sec_class}), not to have to worry about cyclicity of $\eta$.
\end{remark}

To follow Procedure \ref{proc}, we have to face the following questions, which are all related to each other:
\begin{enumerate}
\item \label{Q1}
Can a given pair $(\pi,\eta)$ be completed to a Sch{\"u}rmann triple?

\item \label{Q2}
If $(\pi,\eta,\psi)$ is a Sch{\"u}rmann triple, how far can $\psi$ be away from a generating functional?

\item \label{Q3}
Given two generating functionals $\psi_1$ and $\psi_2$ such that both $(\pi,\eta,\psi_1)$ and $(\pi,\eta,\psi_2)$ are Sch{\"u}rmann triples, how different can $\psi_1$ and $\psi_2$ be?
\end{enumerate}

\noindent
All three questions root in the single question, given a pair $(\pi,\eta)$, what is fixed by the information/wish that a functional $\psi$ completes the pair to a Sch{\"u}rmann triple? The coboundary property in \eqref{**} fixes the values of $\psi$ on $K_2$ (by the line preceding \eqref{**}) and it fixes $\psi$ to be $0$-normalized (simply plug in $a=\U=b$ into \eqref{*} and \eqref{**} to see that $\eta(\U)=0$ and, hence, $\psi(\U)=0$).

The first question is about existence: Does the prescription $ab\mapsto\langle\eta(a^*),\eta(b)\rangle$ (well)define a linear map on $K_2$ (that may, then, be extended further, taking also into account $\psi(\U)=0$, to all of $\cG$)? It has a nontrivial answer, which depends on the quantum semigroup in question. We come back to it, later.

The second and third question are about uniqueness. In both cases there exists a linear functional on $\cG$ fulfilling \eqref{**}, but we wish to know more about the remaining degrees of freedom. Note that a linear functional $\psi$ on $\cG$ satisfying \eqref{**}, is, in particular, conditionally positive.
So, the only question
to be answered for knowing if $\psi$ is a generating functional,
is whether it is hermitian (this is automatic on $K_2$).

We are faced with a simple problem of linear algebra: Given a linear functional on $K_2$ in how many ways can it be extended to a linear functional on $\cG$? In particular, given a Sch{\"u}rmann triple $(\pi,\eta,\psi)$, can we choose the extension to $\cG$ of $\psi\upharpoonright K_2$ in such a way that it becomes a generating functional (Question \eqref{Q2}) and in how many different ways is this possible (Question  \eqref{Q3})?

Let us fix some notation.

\begin{defi}
Let $V$ be a vector space and let $K$ be a subspace of $V$. We say, a family $E$ of vectors (necessarily in $V\backslash K$) is a \hl{basis extension from $K$ to $V$} if $E$ extends one basis of $K$ (and, therefore, all bases of $K$) to a basis of $V$. Equivalently (see also Observation \ref{ideob}), $E$ is a basis extension if  every element of $V$ can be written as the sum of a unique linear combination of vectors in $E$ and a unique element $k\in K$.

(Despite basis extensions $E$
{being},
like bases, families of vectors, frequently we shall be sloppy and consider $E$ just as a set. The only significant difference occurs if the family $E$ has double elements -- in which case it would not be a basis extension.)
\end{defi}

Without proof we state the following lemma from linear algebra.

\begin{lemma} \label{PElem}
For each basis extension $E$ from $K$ to $V$ there is exactly one projection $\cP_E$ from $V$ onto $K$ fulfilling {\normalfont$\ker\cP_E=\ls\,E$}. This projection has the form
$$
\cP_E
~=~
\id-\sum_{\vk\in E}\vk\e'_\vk,
$$
where the linear functionals $\e'_\vk$ are (well)defined by putting $\e'_\vk(\vk')=\delta_{\vk,\vk'}$ for $\vk'\in E$ and $\e'_\vk(k)=0$ for $k\in K$.
\end{lemma}

We follow Skeide \cite{skeide99} and improve \cite{skeide99} quite a bit. Obviously, $\id-\U\e$, our projection onto $K_1=\ker\e$, is the projection associated with the basis extension $\{\U\}$ from $K_1$ to $\cG$, and $\e'_\U$ is just $\e$. By the lemma,
$$
{\cP_{\{1\}}}
~=~
\id-\U\e
$$
is the unique projection onto $K_1$ that has kernel $\U\C$. A linear functional $\psi$ is $0$-normalized if and only $\psi\circ(\id-\U\e)=\psi$.

Let $E_1$ be a basis extension from $K_2$ to $K_1$, so that $\{\U\}\cup E_1$ is a basis extension from $K_2$ to $\cG$. A moments thought (taking also into account that $\{\U\}$, $K_1$, and $K_2$ are $*$-invariant) shows that we may (and, usually, will) assume that $E_1$ is \hl{hermitian}, that is, $E_1$ consists of self-adjoint elements. Then, also the functionals $\e'_\vk$ ($\vk\in E_1$) are hermitian. The associated (hermitian) projection from $\cG$ onto $K_2$ is
\begin{equation} \label{P-form}
\cP
~:=~
\id-\U\e-\sum_{\vk\in E_1}\vk\e'_\vk,
~=~
\cP_{E_1}\circ(\id-\U\e).
\end{equation}
Since it is the linear hull of the elements in a basis extension that determines the projection, we see that among all (hermitian) projections from $\cG$ onto $K_2$, the projections that have the preceding form are exactly those that satisfy the additional condition that $\ker\cP\ni\U$. For such a projection $\cP$, a linear functional satisfying $\psi\circ\cP=\psi$, hence, $\psi\circ(\id-\U\e)=\psi$ is $0$-normalized. And since $\cP$ is hermitian, $\psi\circ\cP=\psi$ is hermitian if and only if $\psi$ is hermitian on $K_2$.

\begin{defi}
A Sch{\"u}rmann triple $(\pi,\eta,\psi)$ for $\psi$ is \hl{trivial} if $\pi=0$. (Then $D=\{0\}$ and $\eta=0$.)
\end{defi}

Obviously, the Sch{\"u}rmann triple of a generating functional $\psi$ is trivial if and only if $\psi\upharpoonright K_2=0$. Generating functionals with trivial Sch{\"u}rmann triple are also called \hl{drifts}. The following statements are fairly obvious; they answer Questions \eqref{Q2} and \eqref{Q3}:
\begin{itemize}
\item
The drifts are precisely the real linear combinations of the functionals $\e'_\vk$ ($\vk\in E_1$).

\item
Two linear functionals $\psi_1$ and $\psi_2$ coincide on $K_2$ (this includes, in particular, two functionals completing the same pair $(\pi,\eta)$ to a Sch{\"u}rmann triple) if and only if they differ by a linear combination of the functionals $\e$, $\e'_\vk$ ($\vk\in E_1$).

\item
In particular, two generating functionals complete the same pair $(\pi,\eta)$ to a Sch{\"u}rmann triple if and only if they differ by a drift.
\end{itemize}
So, the answer to Question \eqref{Q1} is positive, if (and only if) for the pair $(\pi,\eta)$ Equation \eqref{**} (well)defines a linear functional on $K_2$. Given such a functional, we extend it to $\cG$ by putting it $0$ at $\U$ and at all $\vk\in E_1$ (so that $\psi\circ\cP=\psi$),
obtaining a generating functional $\psi$ turning $(\pi,\eta,\psi)$ into a Sch{\"u}rmann triple. Invoking our answer to Question \eqref{Q3}, we obtain all generating functionals doing the same job, by adding to $\psi$ any drift.

To get a handier formulation, we repeat a statement from the end of \cite[Subsection 2.2]{skeide99}.

\begin{prop}
For every generating functional $\psi$ on $\cG$ that is not a drift there is a projection $\cP$ of the form \eqref{P-form} fulfilling
$$
\psi\circ\cP
~=~
\psi.
$$
Every other generating functional $\psi'$ having the same pair $(\pi,\eta)$ in its Sch{\"u}rmann triple, is obtained from $\psi$ as $\psi'=\psi\circ\cP'$ where $\cP'$ is the some projection fulfilling $\psi'\circ\cP'=\psi'$.
\end{prop}

\medskip
\begin{cor} \label{pPpcor}
If $(\pi,\eta,\psi)$ is a nontrivial Sch{\"u}rmann triple for $\psi$, then $\psi$ is a generating functional if and only if the unique projection $\cP$ onto $K_2$ such that $\psi\circ\cP=\psi$ is hermitian. 
\end{cor}

In order to perform Procedure \ref{proc}, we will have to find a suitable choice for $E_1$. This involves, first, the problem to show that the elements $\vk\in E_1$ are enough to span together with $\U$ and $K_2$ everything and, then, to show that they are linearly independent. While the functionals $\e'_\vk$ cannot be defined in the prescribed way before we actually know that the $\vk$ {\it are} linearly independent, in applications (including the present one, but not only) we actually, first, (well)define the $\e'_\vk$ in a different way and, then, {\it use} them to prove that the $\vk$ are linearly independent.

The way
{we}
will define in Section \ref{sec_gaussian} the $\e'_\vk$ for $SU_q(N)$, also explaining the notation, comes from the observation in Skeide \cite[Example 2.2]{skeide99} that every drift $d$ (like, for instance, $\e_\vk$) can be obtained as derivative of the convolution semigroup of characters $e_\star^{td}$. (Being a drift is exactly what makes this a convolution semigroup of characters.) This is one of  two places in these notes where the comultiplication of $\cG$ is, though not strictly necessary, useful. If we, as we plan, forget about the comultiplication, then just the insight that drifts can be obtained by taking derivatives of (suitably parametrized) families of characters remains.

We summarize (basically, \cite[Example 2.2]{skeide99}):

\begin{prop} \label{Gdelprop}
Let $(\e_\theta)_{\theta\ge0}$ be a family of characters with $\e_0=\e$.
\begin{enumerate}
\item
If $e_\theta$ is pointwise differentiable at $\theta=0$, then $\e'_0$ is an $\e$-$\e$-cocycle.

\item
If $e_\theta$ is pointwise twice differentiable at $\theta=0$, then the $\e$-$\e$-$2$coboundary of $\frac{\e''_0}{2}$ is $\e'_0$-induced.
\end{enumerate}
\end{prop}

\noindent
Clearly, $\e'_0$ is $0$ on $K_2\cup\{\U\}$.

\begin{cor} \label{Gdelcor}
Suppose we have
{ a family $(\e'_i)_{i\in I}$ of $\e$-$\e$-cocycles},
all obtained (for suitable families $(\e^i_\theta)_{\theta\ge0}$) as in the proposition, and we have
{ a family $(\vk_i)_{i\in I}$, indexed by the same set $I$},
such that $\e'_j(\vk_i)=\delta_{i,j}$.
{ Then}
the $\vk_i\in K_1\backslash K_2$ are linearly independent and may be extended to a basis extension from $K_2$ to $K_1$.

If, moreover, the $\vk_i$ and $K_2$ generate $K_1$ then the $\vk_i$ are a basis extension from $K_2$ to $K_1$.
\end{cor}

\medskip
\subsection{Generalities about decomposition and approximation} \label{sec_G_D_A}

The results in this subsection address the decomposition of a generating functional into a sum of two and how to (well)define functionals by suitably approximating their cocycles.

The decomposition of a generating functional $\psi$ into a sum $\psi=\psi_1+\psi_2$ of two generating functionals $\psi_1$ and $\psi_2$ is related to direct sum operations among the GNS-representation and GNS-cocycle of the latter two.

\begin{prop} \label{Stsumprop}
Let $\psi$, $\psi_1$, and $\psi_2$ be generating functionals such that $\psi=\psi_1+\psi_2$, and denote by $(\pi_i,\eta_i,\psi_i)$ the Sch{\"u}rmann triples of $\psi_i$ (with pre-Hilbert spaces $D_i$). Then $\eta:=\eta_1\oplus\eta_2\colon b\mapsto\eta_1(b)\oplus\eta_2(b)\in D:=D_1\oplus D_2$ is a cocycle with respect to $\pi:=\pi_1\oplus\pi_2$ and $(\pi,\eta,\psi)$ is a Sch{\"u}rmann triple.
\end{prop}

We omit the obvious/simple proof. Observe, however, that $(\pi,\eta,\psi)$ is, in general, not {\bf the} Sch{\"u}rmann triple of $\psi$. (Just take $\psi_2=\psi_1$. Then the Sch{\"u}rmann triple of $\psi$ is $(\pi_1,\sqrt{2}\eta_1,\psi)$. This problem is the same as for the GNS-representation of the sum of two
positive functionals.) This is one of the main reasons why it would be extremely inconvenient for us to restrict our attention to cyclic cocycles, only.

Seeking a sort of converse of Proposition \ref{Stsumprop}, we observe that cocycles behave nicely with respect to decomposition of the representation space.

\begin{prop}\label{decomprop}
Suppose $\pi_1$ and $\pi_2$ are unital $*$-representations of $\cG$ on pre-Hilbert spaces $D_1$ and $D_2$, respectively, and suppose $\eta$ is a cocycle with respect to $\pi:=\pi_1\oplus\pi_2$.
{
Then, in the unique decomposition $\eta=\eta_1\oplus\eta_2$, the $\eta_i$ are (unique) cocycles with respect to $\pi_i$.}
\end{prop}

We, again, omit the simple proof  that follows simply by observing that $\pi(\U):=\pi_1(\U)\oplus\pi_2(\U)=\id_{D_1\oplus D_2}$, so that the maps $\eta_i:=\pi_i(\U)\eta$ do the job.

\begin{cor}\label{decomcor}
Suppose we have linear functionals $\psi_1$, $\psi_2$, and $\psi$ satisfying $\psi=\psi_1+\psi_2$. If, in the situation of Proposition \ref{decomprop}, two of the triples $(\pi_1,\eta_1,\psi_1)$, $(\pi_2,\eta_2,\psi_2)$, and $(\pi,\eta,\psi)$ are Sch{\"u}rmann triples, then so is the third. If, moreover, two of the functionals are generating, then so is the third.
\end{cor}

Again, there is not really 
anything
to prove. (If two of the three conditions $\langle\eta_i(a^*),\eta_i(b)\rangle=\psi_i(ab)$ and $\langle\eta(a^*),\eta(b)\rangle=\psi(ab)$ are satisfied for all $a,b\in K_1$, then so is the third.) So, this corollary looks rather innocent. It is, however, a surprisingly crucial tool. A L{\'e}vy-Khintchine formula is, in the first place, a decomposition result for functionals, asserting (when true) that all generating functionals can be written as a sum of functionals from two simpler classes. What we can do easily, is decomposing the GNS-representation, hence (by Proposition \ref{decomprop}), the cocycles of a generating functional into direct summands corresponding to the simpler classes. We are left with the (tricky!) question, whether the corresponding components of the cocycle give rise individually to generating functionals (summing up to the original one). The corollary tells: Yes, if we can guarantee existence of a functional for one of the two components of the cocycle; and this is what we will do. For how we are going to do that, the following simple approximation result is of outstanding importance: It tells that if we can approximate the cocycle in the pair $(\pi,\eta)$ by {\it coboundaries} for $\pi$, then there is also a generating functional $\psi$ completing the Sch{\"u}rmann triple $(\pi,\eta,\psi)$.

Recall that for each $*$-representation $\pi$ of $\cG$ on $D$ and each vector $\eta_\U$, the map $\eta:=(\pi\eta_\U)\circ(\id-\U\e)\colon b\mapsto\pi(b-\U\e(b))\eta_\U$ is a $\pi$-$\e$-cocycle. We say, a cocycle of this form is a \hl{coboundary for $\pi$} or just a \hl{coboundary}. We fix $D$ and $\pi$.

Recall, too, that we assume fixed the (hermitian!) projection $\cP$ onto $K_2$. (Subsection \ref{sec_S_trip}.)

\begin{lemma} \label{cocaGFlem}
Suppose
{$(\eta_n)_{n\in\mathbb{N}}$}
is a sequence of vectors in $D$ such that the sequence
{composed of the}
coboundaries  $(\pi\eta_n)\circ(\id-\U\e)$ for $\pi$ converges pointwise to a map, say, $\eta$. Then:
\begin{enumerate}
\item
$\langle\eta_n,\pi(\bullet)\eta_n\rangle\circ\cP$ converges pointwise to a map, say, $\psi$.

\item
$(\pi,\eta,\psi)$ is a Sch{\"u}rmann triple.

\item
$\psi\circ\cP=\psi$, so that, by Corollary \ref{pPpcor}, $\psi$ is a generating functional.
\end{enumerate}
\end{lemma}

\begin{proof}
$\cP$ maps into $K_2$ and a typical element of $K_2$ has the form $a^*b$ for $a,b\in K_1$. So, $\langle\eta_n,\pi(\bullet)\eta_n\rangle\circ\cP(a^*b)=\langle\pi(a)\eta_n,\pi(b)\eta_n\rangle$, which converges on each side, as $\pi(b)\eta_n$ is just the value of the coboundary generated by $\eta_n$. By the same computation, $\psi$ fulfills \eqref{**} and, clearly, $\eta$ fulfills \eqref{*}, so $(\pi,\eta,\psi)$ is a Sch{\"u}rmann triple. The third statement is obvious.
\end{proof}

\medskip
\begin{remark} \label{JPrem}
It is the property that $\psi$ is the limit of expressions $\langle\eta_n,\pi(\bullet)\eta_n\rangle\circ\cP$ which we require in order to call $\psi$ a \hl{jump part}. Every generating functional $\psi=\psi\circ\cP$ is the (pointwise) limit of
$$
\frac{\vp_t-\e}{t}\circ\cP
~=~
\frac{\vp_t}{t}\circ\cP
$$
for $t\to0$, where $\vp_t=e_\star^{\psi t}$ is the generated convolution semigroup 
of states; every $\vp_t$ has a GNS-construction $(\pi_t,\eta_t)$ with the cyclic (unit) vector $\eta_t$ in the representation space $H_t$. What makes the sequence of positive functionals $\langle\eta_n,\pi(\bullet)\eta_n\rangle$ different from the sequence $n\vp_{\frac{1}{n}}$, is the fact that in the former the representation $\pi$ is fixed, while only the vector $\eta_n$ is running.

This is exactly the situation we have in the integral $\int_{{\mathcal G}\backslash\{e\}}f(g)\,d{\mathcal L}(g)$ (see the introduction), if we approximate it by
$$
\int_{{\mathcal G}\backslash U_{\frac{1}{n}}(e)}f(g)\,d{\mathcal L}(g),
$$
where $U_{\delta}(g)$ is (for some metric on the classical group $\cG$) the open $\delta$-neighbourhood of the point $g\in\cG$. Here, $H=L^2(\cG,{\mathcal L})$, the representation $\pi$ of $f$ is by multiplication with $f$, and $\eta_n$ is the indicator function of ${\mathcal G}\backslash U_{\frac{1}{n}}(e)$.
\end{remark}

\begin{ob} \label{jCNGob}
In view of Subsection \ref{ssec_gaussian} (and the terminology introduced there), we mention that if $\pi$ has a gaussian part, then this part disappears under $\pi\circ\cP$. Therefore, as a limit of functionals $\langle\eta_n,\pi(\bullet)\eta_n\rangle\circ\cP$, a jump part is completely non-gaussian.
\end{ob}

\medskip
\subsection{Sch{\"u}rmann triples on quantum subsemigroups} \label{sec_Stqsub}

In the course of proving our results for $SU_q(N)$, we will decompose representations into components that ``live'' on the quantum subgroups $SU_q(n)$ $(n\le N)$.
Also, promoting our results about $SU_q(N)$ into results about $U_q(N)$, is based on the fact that $U_q(N)$ sits in between $SU_q(N)$ and $SU_q(N+1)$.

\begin{defi} \label{def_living}
A \hl{quantum subsemigroup} of a quantum semigroup $\cG$ is a pair $(\cH,s)$ consisting of a quantum semigroup $\cH$ and a surjective 
$*$-bialgebra homomorphism $s\colon\cG \to \cH$.

We say that a map $T$ from 
{$\cG$}
to some space $X$ \hl{lives on} $(\cH,s)$ if $T$ \hl{factors through $s$}, that is, if
there exist a map $\tilde{T}$ from $\cH$ to $X$ such that
\[
T = \tilde{T}\circ s. 
\]
If it is clear from the context what $s$ is, then we will just speak of the quantum subsemigroup $\cH$ of $\cG$.
\end{defi}

\begin{remark}
If the quantum group $\cH$ is a quantum subsemigroup of the quantum group $\cG$ via $s$, then (see Dascalescu, Nastasescu, and Raianu \cite[Proposition 4.2.5]{dascalescu+al}) $s$ also respects the antipodes. That is, $\cH$ is a \hl{quantum subgroup} of $\cG$.
\end{remark}

Since $s$ is surjective, the map $\tilde{T}$ illustrating that $T$ lives on $\cH$ is unique.

In the remainder of this subsection we fix a quantum semigroup $\cG$, one of its quantum subsemigroups $(\cH,s)$, and a (pre-)Hilbert space $D$. Since $s$ 
respects the counits, $\e$ lives on $\cH$ via the counit $\tilde{\e}$ of $\cH$. Also, $K_n$ $(n=0,1,\ldots,n,\ldots,\infty)$ are mapped by $s$ onto the corresponding $\tilde{K}_n$ of $\cH$.

\begin{prop}
Suppose we have maps $\pi$, $\eta$, and $\psi$, all defined on $\cG$, and maps $\tilde{\pi}$, $\tilde{\eta}$, and $\tilde{\psi}$ such that
\begin{align*}
\pi
&
~=~
\tilde{\pi}\circ s,
&
\eta
&
~=~
\tilde{\eta}\circ s,
&
\psi
&
~=~
\tilde{\psi}\circ s.
\end{align*}

\vspace{-2.5ex}
\noindent
Then:
\vspace{-2ex}
\begin{enumerate}
\item
$\pi$ is a $*$-representation of $\cG$ (obviously, living on $\cH$) if and only if $\tilde{\pi}$ is a $*$-represen\-tation of $\cH$.

\item
$\eta$ is a $\pi$-$\e$-cocycle (obviously, living on $\cH$) if and only if $\tilde{\eta}$ is a $\tilde{\pi}$-$\tilde{\e}$-cocycle.

\item
$\psi$ is a generating functional on $\cG$ (obviously, living on $\cH$) if and only if $\tilde{\psi}$ is a generating functional on $\cH$.

\item
$(\pi,\eta,\psi)$ is a Sch{\"u}rmann triple if and only if $(\tilde{\pi},\tilde{\eta},\tilde{\psi})$ is a Sch{\"u}rmann triple.
\end{enumerate}
\end{prop}

\begin{proof}
The {\it if}-direction is clear, while the {\it only if}-direction follows from $s(K_n)=\tilde{K}_n$.\end{proof}

Of course, the projections onto $K_1$ are compatible in the sense that
$$
(\id-\U\tilde{\e})\circ s
~=~
s\circ(\id-\U\e).
$$

\begin{cor} \label{subapprcor}
Suppose a $\pi$-$\e$-cocycle $\eta$ on $\cG$ can be approximated by coboundaries $(\pi\eta_n)\circ(\id-\U\e)$. If $\pi=\tilde{\pi}\circ s$ lives on $\cH$, then
$$
\tilde{\eta}
~:=~
\lim_{n\to\infty}(\tilde{\pi}\eta_n)\circ(\id-\U\tilde{\e})
$$
exists (pointwise), is a $\tilde{\pi}$-$\tilde{\e}$-cocycle, and fulfills $\eta=\tilde{\eta}\circ s$, so that $\eta$ lives on $\cH$, too.

Moreover, denoting $\psi$ as in Lemma \ref{cocaGFlem} and $\tilde{\psi}$ the analogue for $\tilde{\eta}$ and a projection $\cP^\cH$ onto $\tilde{K}_2$,
we get Sch{\"u}rmann triples $(\pi,\eta,\psi)$ and $(\tilde{\pi},\tilde{\eta},\tilde{\psi})$, and $\psi$ and $\tilde{\psi}$ are related by
$$
\psi
~=~
\tilde{\psi}\circ s\circ\cP.
$$
\end{cor}

\begin{proof}
Almost everything follows, appealing to surjectivity of $s$, by writing arguments of ``twiddled'' maps in the form $s(a)$. The only thing that needs a word, is the last formula. Clearly, $\tilde{\psi}\circ s$ defines a generating functional on $\cG$ (that lives on $\cH$). By
$$
(\tilde{\pi}\eta_n)\circ(\id-\U\tilde{\e})\circ s
~=~
(\pi\eta_n)\circ(\id-\U\e),
$$
we see that $\tilde{\psi}\circ s$ coincides with $\psi$ on $K_2$. From that, the formula follows.\end{proof}

We may ask, whether the projections $\cP$ and $\cP^\cH$ may be chosen compatible, too, so that in the last formula we really get $\psi=\tilde{\psi}\circ s$, without the correction of the drift part via composition with $\cP$. The answer is yes, as long as the $\cG$ and its subsemigroup $\cH$ are fixed; but the possible choices of $\cP$ depend on $\cH$.

\begin{lemma} \label{Psublem}
Let $E_1$ be the basis extension from $K_2$ to $K_1$ determining $\cP$ and let $E^\cH_1$ be the basis extension from $\tilde{K}_2$ 
to $\tilde{K}_1$ determining $\cP^\cH$. Then
$$
\cP^\cH\circ s
~=~
s\circ\cP
$$
if and only if {\normalfont $\,\ls\,s(E_1)\subset\,\ls\,E^\cH_1$}.

Moreover, if {\normalfont $\,\ls\,s(E_1)\subset\,\ls\,E^\cH_1$}, then necessarily {\normalfont $\,\ls\,s(E_1)=\,\ls\,E^\cH_1$}.
\end{lemma}

\begin{proof}
First of all, since $\cP$ maps into $K_2$, since $s$ maps $K_2$ into (actually, onto) $\tilde{K}_2$, and since $\cP^\cH$ acts as identity on $\tilde{K}_2$, the right-hand side always coincides with $\cP^\cH\circ s\circ\cP$. Therefore, we have to examine, when $\cP^\cH\circ s$ coincides with $\cP^\cH\circ s\circ\cP$. Since $\cP$ acts as identity on
$K_2$,
the two maps always coincide on $K_2$. Since $K_2$ and $E_1$ span $K_1$, it remains to examine, when the two maps coincide on $\,\ls\, E_1$.

By Lemma \ref{PElem}, the restriction of $\cP$ to $K_1$ is the unique idempotent onto $K_2$ with kernel $\,\ls\, E_1$. Therefore, $\cP^\cH\circ s\circ\cP$ is $0$ on $\,\ls\, E_1$. Likewise, the restriction of $\cP^\cH$ to $\tilde{K}_1$ is the unique idempotent onto $\tilde{K}_2$ with kernel $\,\ls\,E^\cH_1$. Therefore, $\cP^\cH\circ s$ is $0$ on $\,\ls\, E_1$ if and only if $\,\ls\,s(E_1)\subset\,\ls\,E^\cH_1$.

As for the last statement, since $E_1$ and $K_2$ span $K_1$ and since $s$ is surjective, for any $\tilde{\vk}\in E^\cH_1$ there are $k_1\in\,\ls\,E_1$ and $k_2\in K_2$ such that $s(k_1)+s(k_2)=\tilde{\vk}$. Plugging this in into $\cP^\cH$, taking also into account that $\tilde{\vk}\in\,\ls\,E^\cH_1=\ker\cP^\cH$, under the hypothesis $\,\ls\,s(E_1)\subset\,\ls\,E^\cH_1$ we get
$$
0
~=~
\cP^\cH(\tilde{\vk})
~=~
\cP^\cH\circ s(k_1)+\cP^\cH\circ s(k_2)
~=~
0+s(k_2).
$$
It follows that $\tilde{\vk}=s(k_1)\in\,\ls\,s(E_1)$. Therefore, $\,\ls\,E^\cH_1\subset\,\ls\,s(E_1)$, too.\end{proof}

\medskip
\begin{cor} \label{subPlcor}
If   {\normalfont $\,\ls\,s(E_1)\subset\,\ls\,E^\cH_1$}, then the generating functional $\psi$ from Corollary \ref{subapprcor} lives on $\cH$.
\end{cor}

\begin{cor} \label{subPcor}
Suppose $s(E_1)=E^\cH_1\cup\{0\}$. Then $\cP^\cH\circ s=s\circ\cP$.
\end{cor}

Corollary \ref{subapprcor} will develop its full power only in Section \ref{sec_uqn}, when we reduce the case $U_q(N)$ to the case $SU_q(N)$.
In the decomposition of generating functionals on $SU_q(N)$ into components that live on the subgroups $SU_q(n)$ $(n\le N)$, Corollary \ref{subapprcor} does not help. We do get a decomposition of the representations into representations that live on the subgroups (Subsection \ref{repdSSEC}). We also can show that (for $n\ge2$) the corresponding components of the cocycles are limits of coboundaries (Subsection \ref{cocaSSEC}). But, before we can show the latter statement, we first have to show that the components of the cocycles live on the subgroups without knowing they are limits of coboundaries (Subsection \ref{cocdSSEC}).

This leaves us with the general problem to check when maps $\pi$ and $\eta$ on $\cG$ do live on $\cH$. Our algebras are generated (as algebras) by sets of generators, usually, arranged in matrices, and dividing out
{some relations} on these generators.
Already when (well)defining representations and their cocycles, we are using the well-known facts that a representation (or, more generally, a homomorphism) is well-defined by assigning its values on the generators, and checking whether the assigned values satisfy the relations. The same is true for a cocycle.

All our quantum subgroups arise by adding more relations. Therefore, the ($*$-)algebra describing the quantum subgroup can be obtained by taking the quotient of the ($*$-)algebra describing the containing quantum group and the ($*$-)ideal generated by the extra relations. The following, purely algebraic, lemma tells us what we have to do to check if a representation and its cocycle live on a quantum subsemigroup. (Applying it to the free algebra and its quotients, we also obtain a proof of the preceding statements about representations and cocycles on algebras generated by relations.)

\begin{lemma} \label{exrellem}
Let $\A$ be a unital algebra over $\C$ and let $\e$ be a homomorphism into $\C$. Suppose $R=\{r_1,\ldots,r_M\}$ is a subset of $\ker\e$, let {\normalfont $I:=\ls\A R\A$} denote the ideal generated by $R$, and denote by $s$ the canonical homomorphism $\A\rightarrow\A/I$. Then:

A representation (or homomorphism) $\pi$ of $\A$ defines a representation (or homomorphism) $\tilde{\pi}\colon s(a)\mapsto\pi(a)$ of $\A/I$ if and only if $\pi(R)=\{0\}$.

If $\pi(R)=\{0\}$, then a $\pi$-$\e$-cocycle $\eta$ defines a $\tilde{\pi}$-$\tilde{\e}$-cocycle $\tilde{\eta}\colon s(a)\mapsto\eta(a)$ if and only if $\eta(R)=\{0\}$.
\end{lemma}

\begin{proof}
We have to show that $\pi(a)=0$ for all $a\in I$. Clearly, by $\pi(br_kc)=\pi(b)\pi(r_k)\pi(c)$, this is fulfilled if and (recall that representations are assumed unital!) only if $\pi(R)=\{0\}$.

In particular, $\e$ satisfies the condition, so there is $\tilde{\e}$.

Finally, we have to show that  $\eta(a)=0$ for all $a\in I$. Clearly, by $\eta(br_kc)=\pi(b)\pi(r_k)\eta(c)+\pi(b)\eta(r_k)\e(c)+\eta(b)\e(r_k)\e(c)$, this is fulfilled if and only if $\eta(R)=\{0\}$.
\end{proof}

And for making this lemma more applicable in our context:

\begin{cor} \label{exrelcor}
Under the hypotheses of Lemma \ref{exrellem}, suppose that $S$ is a subset of $\A/I$ and denote by $J=\ls(\A/I)S(\A/I)$ the ideal in $\A/I$ generated by $S$. For each $s\in S$ choose $a_s\in\A$ such that $a_s+I=s$ and denote by $K$ the ideal in $\A$ generated by $R\cup\{a_s\colon s\in S\}$. Then
$$
(a+I)+J
~\longmapsto~
a+K
$$
defines an isomorphism $(\A/I)/J\rightarrow\A/K$.
\end{cor}

\begin{proof}
Since $R\subset R\cup\{a_s\colon s\in S\}\subset K$, the canonical homomorphism $\pi\colon a\mapsto a+K$ vanishes on $R$, hence, defines a homomorphism $\tilde{\pi}\colon a+I\mapsto a+K$. Since $\{a_s\colon s\in S\}\subset R\cup\{a_s\colon s\in S\}\subset K$ and since $S=\{a_s+I\colon s\in S\}$, the homomorphism $\tilde{\pi}$ vanishes on $S$, hence, defines a homomorphism $\tilde{\tilde{\pi}}\colon(a+I)+J\mapsto a+K$. Conversely, the homomorphism $a\mapsto (a+I)+J$ defines a homomorphism $\A/K\rightarrow(\A+I)+J$, obviously the inverse of $\tilde{\tilde{\pi}}$, because $(r+I)+J=0$ for all $r\in R$ and because $(a_s+I)+J=s+J=0$ for all $s\in S$.
\end{proof}

\medskip
\subsection{Gaussian generating functionals and L\'evy-Khintchine decomposition}
\label{ssec_gaussian}

In the classical theory of L{\'e}vy processes with values in abelian Lie groups, the coefficient algebra can be thought of as polynomials in the coordinate functions. The ideals $K_n$ correspond to polynomials with expansion starting with monomials of degree at least $n$. Correspondingly, functionals vanishing on $K_n$ may be viewed as having no contribution on $n$th and higher powers. We have met the {\it drifts} that vanish on $K_2$; they may, therefore, be called as {\it linear}. Consequently, the functionals that vanish on $K_3$ may be referred to as {\it quadratic}. Quadratic generating functionals, in the classical theory, correspond to second order differential operators. They generate {\it Brownian motions}, and are called {\it gaussian}. A crucial part of the classical L{\'e}vy-Khintchine formula consists in splitting an arbitrary generating functional into a (maximal) gaussian part and a residue part (with no gaussian component remaining). The residue part, classically is a (topological) convex combination of generating functionals of pure jump processes. (On $\R$ we get a {\it compound Poisson process}.) It is, therefore, sometimes referred to as {\it jump part}. (Recall Remark \ref{JPrem}.)

In this subsection (following Sch{\"u}rmann \cite{MSchue90b} and Skeide \cite{skeide99}), we discuss gaussian parts and
{explain}
what we expect from a L{\'e}vy-Khintchine decomposition in the general context of (algebraic) quantum semigroups. We also report results from Franz, Gerhold, and Thom \cite{franz_gerhold_thom}, where the basic problems are discussed and classified.

We have already examined the generating functionals with trivial Sch{\"u}rmann triples, the drifts. The one step less simple class would be Sch{\"u}rmann triples with representations that are multiples $\id_D\e$ of $\e$. From
\begin{align*}
\eta(ab)
&
~=~
\pi(a)\eta(b)
&
\langle\eta(a),\pi(b)\eta(c)\rangle
&
~=~
\psi(a^*bc)
&
(a,b,c\in K_1),
\end{align*}
we conclude:

\begin{prop}
Let  $\psi$ be a generating  functional and let $(\pi,\eta,\psi)$ be its Sch{\"u}rmann triple. Then the following conditions are equivalent:
\begin{enumerate}
\item
$\psi$ vanishes on $K_3$.

\item
$\eta$ vanishes on $K_2$.

\item
$\pi$ vanishes on $K_1$.
\end{enumerate}
\end{prop}

\noindent
Recall that $\eta$ in {\bf the} Sch{\"u}rmann triple is cyclic. Without that, $\pi$ may fail to vanish on $K_1$, even if the other two, still equivalent, conditions are satisfied.

\begin{defi}
A generating functional, a cocycle, and a representation on a  quantum semigroup are called \hl{gaussian} if they vanish on $K_3$, $K_2$, and $K_1$, respectively.
\end{defi}

From
$$
\pi
~=~
\pi\circ(\id-\U\e)+\pi(\U)\e,
$$
we conclude:

\begin{prop}
A (unital!) $*$-representation $\pi$ is gaussian if and only $\pi=\id_D\e$.
\end{prop}

\noindent
Recall that we fixed a hermitian basis extension $E_1$ from $K_2$ to $K_1$, and that we have the functionals $\e'_\vk$ ($\vk\in E_1$) as in Lemma \ref{PElem}.

\begin{prop} \label{Gcocprop}
A cocycle $\eta$ is gaussian if and only if it has the form
$$
\eta
~=~
\sum_{\vk\in E_1}\eta_\vk\e'_\vk
$$
for vectors  $\eta_\vk\in D$. The vectors are unique; in fact, $\eta_\vk=\eta(\vk)$.
\end{prop}

\begin{proof}
Express $a\in\cG$ as $\U\e(a)+k+\sum_{\vk\in E_1}\vk\e'_\vk(a)$. Such a decomposition is unique and, necessarily, fulfills $k\in K_2$.
Taking into account that a gaussian $\eta$ vanishes on $\U$ and on $K_2$, the formula follows. The other statements are obvious.
\end{proof}

We described gaussian representations in an easy and concise way. We described gaussian cocycles in a similarly easy and concise way, provided we have found a hermitian basis extension $E_1$. It would be desirable to have a similar description of gaussian generating functionals. However, it turns out that the form of a general gaussian generating functional on a quantum semigroup depends on the quantum semigroup. In
{particular}
already the answer to the question, which gaussian cocycles actually {\bf do} admit a gaussian generating functional, does depend on the quantum semigroup. Sch\"urmann \cite[Proposition 5.1.11]{schurmann93} showed that a sufficient (but, in general, not necessary) condition for that a gaussian cocycle $\eta$ admits a generating functional is that it be \hl{hermitian}, that is, $\langle\eta(a^*),\eta(b)\rangle=\langle\eta(b^*),\eta(a)\rangle$. In other words, since the $\vk$ are self-adjoint, the matrix $\langle\eta_\vk,\eta_{\vk'}\rangle$ is real and symmetric. The following little consequence applies to our case $SU_q(N)$; see Lemma \ref{lem_kernel_descr} and its corollary.

\begin{cor} \label{K3gcor}
Suppose that $ab-ba\in K_3$ for all $a,b\in K_1$. Then for a gaussian cocycle $\eta$ there is a Sch{\"u}rmann triple $(\pi,\eta,\psi)$ if and only if $\eta$ is hermitian.
\end{cor}

\begin{proof}
If $(\pi,\eta,\psi)$ is a Sch{\"u}rmann triple for a gaussian cocycle $\eta$, then $\langle\eta(a^*),\eta(b)\rangle=\psi(ab)=\psi(ba)=\langle\eta(b^*),\eta(a)\rangle$ for all $a,b\in K_1$. (The middle equality because, under the stated hypothesis, $\psi(ab-ba)=0$.) The other direction is Sch{\"u}rmann's result.
\end{proof}

\vspace{0ex}
\begin{remark} \label{rem_5.1.11}
Sch{\"u}rmann's proof of \cite[Proposition 5.1.11]{schurmann93} is direct. (It does use the comultiplication.) In the course of arriving at Theorem \ref{Gclassthm}, we also recover Sch{\"u}rmann's result for $SU_q(N)$ (giving an explicit and classifying form to its gaussian generating functionals), appealing to a procedure similar to Proposition \ref{Gdelprop}.
\end{remark}

\medskip
After these generalities about gaussian parts, let us come back to the main problem, decomposition into a gaussian and a {\it completely non-gaussian part}, the first scope of a L{\'e}vy-Khintchine formula. Almost as easy as it is to understand gaussian representations and gaussian cocycles, it is also easy to separate a representation and, accordingly (by means of Proposition \ref{decomprop}), the cocycles into a (maximal) gaussian part and a remaining (completely non-gaussian) part. We report the output of the careful discussion preceding \cite[Definition 2.4]{franz_gerhold_thom}:

\begin{prop} \label{GcnGprop}
Let $(\pi,\eta,\psi)$ be a cyclic Sch{\"u}rmann triple on a pre-Hilbert space $D$ for a generating functional $\psi$. Then there are pre-Hilbert subspaces
{$D_1$ and $D_2$}
of $\overline{D}$ with representation $\pi_1$ and $\pi_2$, respectively, such that:
\begin{itemize}
\item
 $D\subset D_1\oplus D_2(\subset\overline{D})$.
 
 \item
 $\pi=(\pi_1\oplus\pi_2)\upharpoonright D$, so that also $\eta=\eta_1\oplus\eta_2$.
 
 \item
$\eta_1(\cG)=D_1$  and $\pi_1$ is gaussian.
 
 \item
 $\eta_2(\cG)=D_2$  and $\pi_2$ is \hl{completely non-gaussian}, that is, the only invariant subspace of $D_2$ on which $\pi_2$ restricts to a gaussian representation, is the trivial subspace $\{0\}$.
\end{itemize}
\end{prop}

\noindent
One may show that $D_1$ and $D_2$ fulfilling these conditions are unique. (In fact, one necessarily has $D_1=p_1\eta(\cG)$, where $p_1\in B(\overline{D})$ is the projection onto the completion of $D_0=\bigcap_{a\in K_1}\ker\pi(a)$, and $D_2=p_2\eta(\cG)$, where $p_2=\id-p_1$. One verifies that $\pi_i(a)\colon p_i\eta(b)\mapsto p_i\pi(a)\eta(b)$ (well)defines the desired presentations.) Also, $\pi_1$ and $\pi_2$ are \hl{maximally gaussian} and \hl{maximally completely non-gaussian}, respectively.%
\footnote{ \label{FN}
It is noteworthy that the proof of the proposition and the supplementary uniqueness statement do not really gain much simplicity, if the representation operators are bounded (the case that interests us). Therefore, we prefer to formulate here the, not so well-known, general version from \cite{franz_gerhold_thom}, which works for all (algebraic) quantum semigroups.
}

\begin{defi} \label{LKddefi}
A \hl{L{\'e}vy-Khintchine decomposition} for $\psi$ is a pair of generating functionals $\psi_1$ and $\psi_2$ such that $\psi=\psi_1+\psi_2$,
{ and, with the pairs $(\pi_i,\eta_i)$ $(i=1,2)$ as in Proposition \ref{GcnGprop},} 
each $(\pi_i,\eta_i,\psi_i)$ $(i=1,2)$ is a Sch{\"u}rmann triple.

If $\psi$ has L{\'e}vy-Khintchine decomposition, then we put $\psi_G:=\psi_1$ and $\psi_L:=\psi_2$. Given $\cP$, among all L{\'e}vy-Khintchine decompositions $\psi=\psi_G+\psi_L$ there is a unique one satisfying $\psi_L\circ\cP=\psi_L$.
\end{defi}

\medskip
In view of Corollary \ref{decomcor}, for getting a L{\'e}vy-Khintchine decomposition, it is enough to guarantee existence of $\psi_i$ for one $i$. Franz, Gerhold, and Thom \cite{franz_gerhold_thom} have analyzed the corresponding conditions and showed that none of the following four properties are equivalent, though all of them imply the last one.

\medskip
\begin{defi}~
{We say that a quantum semigroup $\cG$ has}
\begin{itemize}
\item
\hl{property} {\bf (AC)}
{if for}
each pair $(\pi,\eta)$ there exists a Sch{\"u}rmann triple $(\pi,\eta,\psi)$.

\item
\hl{property} {\bf (GC)}
{if for}
each gaussian pair $(\pi,\eta)$ there exists a Sch{\"u}rmann triple $(\pi,\eta,\psi)$.

\item
\hl{property} {\bf (NC)}
{if for}
each completey non-gaussian pair $(\pi,\eta)$ there exists a Sch{\"u}rmann triple $(\pi,\eta,\psi)$.

\item
\hl{property} {\bf (LK)}
{if every generating functional $\psi$ on $\cG$}
admits a L{\'e}vy-Khintchine decomposition.
\end{itemize}
\end{defi}

\noindent
One of our main results asserts that $SU_q(N)$ and $U_q(N)$ have Property (NC), hence, (LC). $SU_q(N)$ does not have Property (GC) for $N\ge3$ (showed also, by different means, in Das, Franz, Kula, and Skalski \cite[Proposition 2.3]{dfks18}) and it does have Property (GC), hence,
{also}
(AC) for $N\le2$.
$U_q(1)$ is equal to $U(1)$ and has (AC), while for $N\ge 2$, $U_q(N)$ does only have Property (NC), but not (GC), hence, nor (AC).

\medskip
\subsection{The quantum groups $SU_q(N)$ and $U_q(N)$}
\label{ssec_suqn}

As a compact quantum group, $SU_q(2)$ was introduced by Woronowicz in \cite{Wor87b}, and in \cite{Wor88} he obtained the rest of the family $SU_q(N)$, $N\ge 3$, as an application of a generalization of the Tannaka-Krein duality theorem. Rosso \cite{rosso90} extended these results further to $q$-deformations of other semi-simple compact Lie groups.

Recall that in these notes we forget about the fact that $\cG$ has an antipode. So we are looking at $SU_q(N)$ and $U_q(N)$ rather as quantum semigroups. For simplicity, we also always assume $0<q<1$. (In general, one considers $0<|q|\le 1$, where $q=1$ corresponds to the classical cases. In the, in some sense, degenerate case $q=0$, the antipode is missing. But, see also Number \eqref{two} in Section \ref{sec_final}.) However, $SU_q(N)$ and $U_q(N)$ ($q\ne0$) {\bf are} quantum groups and the various inclusions as quantum subsemigroups they satisfy, are actually inclusions as quantum subgroups. So, we simply trust the literature and will say, from now on, quantum group and quantum subgroup despite
never explicitly
addressing
their antipodes and properties referring to them.

We also take, from now on, into account that (by the unitarity conditions in \eqref{eq_U}, below) all $*$-representations of $\cG$ are by bounded operators  and, therefore, are assumed to act on a Hilbert space, rather than only on pre-Hilbert space. A cocycle $\eta$ with respect to a representation $\pi$ on a Hilbert space $H$ is, therefore, \hl{cyclic} if $\overline{\eta(\cG)}=H$.

Our focus is on the two series of compact quantum groups $SU_q(N)$ and $U_q(N)$ ($N\ge2)$. For formal reasons, we extend all our definitions to $N=1$. (Here $SU_q(1)=SU(1)=\{e\}$, the trivial group, and $U_q(1)=U(1)$, the torus.) All $SU_q(N)$ and $U_q(N)$ are examples of compact quantum \hl{matrix} groups (\hl{CQMG}s \cite{DiKo94} or the topological version \hl{compact matrix pseudo groups} \cite{Wor87a}) of order $N$. This means that $\cG$ is defined as the unital $*$-algebra generated by $N^2$ indeterminates arranged in a matrix $U=[u_{jk}]_{j,k=1}^N$ subject to the \hl{unitarity conditions}
\begin{equation} \label{eq_U}
\sum_{s=1}^{N} u_{js}u_{ks}^* 
~=~
\U\delta_{jk}
~=~
\sum_{s=1}^{N}u_{sj}^*u_{sk}
\end{equation}
and, depending on which quantum matrix group, further relations. The unitarity conditions guarantee that all
{$*$-representations}
map each $u_{jk}$ to a contraction. Comultiplication and counit are defined by extending 
\begin{subequations} \label{eq_epDe}
\begin{align}
\label{eq_Delta}
\Delta
\colon
 u_{jk}
 &~\longmapsto~
\sum_{s=1}^N u_{js}\otimes u_{sk}
\\
\label{eq_epsilon}
\e
\colon
 u_{jk}
&
~\longmapsto~
\delta_{j,k}
\end{align}
\end{subequations}
as unital $*$-homomorphisms. (Of course, for each $\cG$ one has to verify that the assignments $\Delta(u_{jk})$ and $\e(u_{jk})$, respectively, satisfy all the defining relations of $\cG$.)

Let $S_N:=\bigl\{\sigma\colon\sigma\text{ a bijection on $\{1,\ldots,N\}$}\bigr\}$ denote the \hl{symmetric group} of order $N$, and for any $\sigma\in S_N$ denote by $i(\sigma):=\#\{(j,k)\colon j<k,\sigma(j)>\sigma(k)\}$ the \hl{number of inversions} of $\sigma$. For every $\tau\in S_N$ we define
$$
D^q_\tau(U)
~:=~
\sum_{\sigma\in S_N} (-q)^{i(\sigma)}
u_{\sigma(1),\tau(1)}u_{\sigma(2),\tau(2)}\ldots u_{\sigma(N),\tau(N)}.
$$
Usually, $\pol$  is defined by adding to the unitarity conditions in \eqref{eq_U}, the \hl{twisted determinant conditions}
\begin{equation} \label{eq_TD}
D^q_\tau(U)
 ~=~
\U(-q)^{i(\tau)}
\end{equation}
for all $\tau\in S_N$.

Instead of the usual definition (unital $*$-algebra generated by $u_{jk}$ subject to the relations in \eqref{eq_U} and \eqref{eq_TD}), we prefer a different path and follow the exposition in Koelink \cite[Section 2]{koelink91}. We do not start with the unital $*$-algebra generated by the $u_{jk}$ but with the unital algebra, on which, then, an involution is defined. This has the enormous advantage that, here and later on, in (well)defining maps (representations and cocycles), we have to control relations only for the generators $u_{jk}$ but not for their adjoints; the saved amount of time is considerable.

Let us recall that the \hl{quantum determinant} of a matrix $U$ is defined as
$$
D_q(U)
~:=~
\sum_{\sigma\in S_N} (-q)^{i(\sigma)}
u_{1,\sigma(1)}u_{2,\sigma(2)}\ldots u_{N,\sigma(N)}.
$$
The \hl{quantum minor} $D^{jk}_q(U)$ is defined as the quantum determinant of the $(N-1)\times(N-1)$-matrix obtained from the matrix $U$ by removing the $j$-th row and the $k$-th column. That is,
\begin{equation} 
 \label{Dij}
D^{jk}_q(U)
~:=~
\sum_{\sigma \in S_{N-1}^{jk}} (-q)^{i(\sigma)}
u_{1,\sigma(1)}\ldots u_{j-1,\sigma(j-1)}u_{j+1,\sigma(j+1)}\ldots
u_{N,\sigma(N)},
\end{equation}
where $S_{N-1}^{jk}$ denotes the set of bijections $\sigma:\{1, \ldots, j-1,j+1
\ldots, N\}\to\{1, \ldots, k-1,k+1, \ldots, N\}$. We, usually, abbreviate $D:=D_q(U)$ and $D^{jk}:=D^{jk}_q(U)$.

We {\bf define} $U_q(N)$ to be the unital (complex, but not $*$) algebra generated by the $N^2+1$ indeterminates $u_{j,k}$ ($j,k=1,\ldots,N$) and $D^{-1}$ subject to the relations
\begin{subequations} \label{eq_Rq}
\begin{align}
\hspace{15ex}
u_{ij}u_{kj}&~=~ qu_{kj}u_{ij} && \mbox{for} \; i<k, \label{eq_i<k}\\
 u_{ij}u_{il}&~=~ qu_{il}u_{ij} && \mbox{for} \; j<l, \label{eq_j<l}\\
 u_{ij}u_{kl}&~=~ u_{kl}u_{ij} && \mbox{for} \; i<k, j>l, \label{eq_i<k,j>l}\\
 u_{ij}u_{kl} &~=~ u_{kl}u_{ij}-({\textstyle\frac{1}{q}}-q)u_{il}u_{kj} && \mbox{for} \;
i<k, j<l, \label{eq_i<k,j<l}
\hspace{10ex}
\end{align}
and
\begin{equation}
D^{-1}D=\U=DD^{-1}.
\end{equation}
\end{subequations}

Recall that $D$ is, by definition, in the subalgebra generated by the $u_{jk}$ alone, and that,
{ 
by \eqref{eq_i<k} to \eqref{eq_i<k,j<l}},
$D$ in central. (This also means that $U_q(N)$ is isomorphic to the central extension of the algebra generated indeterminates $u_{jk}$ subject to \eqref{eq_Rq},
{ including the extra relation $D^{-1}D=\U=DD^{-1}$.})

One may check that Equations \eqref{eq_epDe} plus $\Delta(D^{-1}):=D^{-1}\otimes D^{-1}$ and $\e(D^{-1}):=1$, turn $U_q(N)$ into a bialgebra. (There is also an antipode.) One may also check that $D=D^q_{\id}(U)$.

Defining
\begin{equation} 
 \label{eq_Du*}
 u_{jk}^* := (-q)^{k-j}D^{jk} D^{-1},
\end{equation}
one may show two things. Firstly, the map
$$
u_{jk}
~\longmapsto~
u_{jk}^*,
\hspace{10ex}
D^{-1}
~\longmapsto~
D
$$
extends to an involution, turning $U_q(N)$ into a $*$-bialgebra (even a Hopf $*$-algebra); this concludes the definition of $U_q(N)$. Secondly, the $u_{jk}$ and $u_{jk}^*$ fulfill the unitarity conditions in \eqref{eq_U}. (To be honest, one, first, verifies that the elements in \eqref{eq_Du*} would satisfy the unitarity condition which, then, motivates to define the involution in that way.)

Now, $\pol$ is {\bf defined} to be the quotient of $U_q(N)$ by the extra relation $D=\U$. Clearly, the homomorphisms $\Delta$ and $\e$ respect this extra relation, so, by Lemma \ref{exrellem}, they survive the quotient. A similar argument shows that the involution survives the quotient, too. By Corollary \ref{exrelcor}, $\pol$ is isomorphic to the algebra generated by indeterminates $u_{jk}$ subject to Relations \eqref{eq_Rq} and $D=\U$. 

\medskip
We briefly explore several homomorphisms, which illustrate how the several $SU_q(n)$ and $U_q(m)$ sit in each other as quantum subgroups. Of course, they reoccur when we examine whether maps live on a quantum subgroup. (See Subsection \ref{sec_Stqsub}.)

By definition, $SU_q(N)$ is a quantum subgroup of $U_q(N)$ via the quotient map
$$
\breve{t}_N\colon
U_q(N)\ni u_{jk}
~\longmapsto~
u_{jk}\in\pol,
\hspace{10ex}
D^{-1}
~\longmapsto~
\U.
$$
But also $U_q(N)$ is a quantum subgroup of $SU_q(N+1)$. Indeed, quite obviously the map
\begin{equation} \label{eq_uq_inclusion}
t_N\colon
\left( \begin{array}{ccccccc} 
u_{11} & \ldots & u_{1N}  & u_{1,N+1} \\
\vdots &  & \vdots & \vdots   \\
u_{N1} & \ldots & u_{NN} & u_{N,N+1}\\
u_{N+1,1} & \ldots & u_{N+1,N}  &u_{N+1,N+1}  \\ 
\end{array} \right)
~\longmapsto~
\left( \begin{array}{ccccccc} 
u_{11} & \ldots & u_{1N}  & 0 \\
\vdots &  & \vdots & \vdots   \\
u_{N1} & \ldots & u_{NN} & 0\\
0 & \ldots & 0  & D^{-1}  \\ 
\end{array} \right)
\end{equation}
determines a surjective homomorphism from $\pol[N+1]$ onto $U_q(N)$ (where, in the same way as in the definition of $\breve{t}_N$ the $u_{jk}$ on the left-hand side are the generators of $\pol[N+1]$ while the $u_{jk}$ on the right-hand side are the generators of $U_q(N)$), and a few computations show that $t_N$ respects comultiplication, counit, and involution. By iterating the $t_N$ and the $\breve{t}_N$ appropriately, we get a chain
$$
SU_q(1)\subset U_q(1)\subset
SU_q(2)\subset U_q(2)\subset
\ldots\subset
SU_q(N)\subset U_q(N)\subset
\ldots
$$
Of particular interest for us is the homomorphism $s_{N}:=\breve{t}_{N-1}\circ t_{N-1}$ given by
\begin{equation} \label{eq_suq_inclusion}
s_N\colon
\left( \begin{array}{ccccccc} 
u_{11} & \ldots & u_{1,N-1}  & u_{1N} \\
\vdots &  & \vdots & \vdots   \\
u_{N-1,1} & \ldots & u_{N-1,N-1} & u_{N-1,N}\\
u_{N1} & \ldots & u_{N,N-1}  &u_{NN}  \\ 
\end{array} \right)
~\longmapsto~
\left( \begin{array}{ccccccc} 
u_{11} & \ldots & u_{1,N-1}  & 0 \\
\vdots &  & \vdots & \vdots   \\
u_{N-1,1} & \ldots & u_{N-1,N-1} & 0\\
0 & \ldots & 0  & \U \\ 
\end{array} \right),
\end{equation}
which establishes the inclusion $SU_q(N-1)\subset SU_q(N)$. This case is so important for us, that we rest for a moment to convince ourselves that this really defines a homomorphism. For this moment, we distinguish between the generators $u_{jk}$ of $\pol$ on the left-hand side, and their images $v_{jk}:=s_N(u_{jk})$ on the right-hand side. (So, for $j,k\le N-1$, the $v_{jk}$ are the generators $u_{jk}$ of $\pol[N-1]$.) Clearly, $s_N$ respects the relations in \eqref{eq_Rq}. (Indeed, those relation that regard only indices not bigger than $N-1$ are fulfilled, because the generators of $\pol[N-1]$ fulfill them. Those relation in \eqref{eq_i<k} and \eqref{eq_j<l} that have at least one index equal to $N$, also contain at least one factor of the type $v_{kN}$ or $v_{Nk}$, hence, are identically $0$. The same is true for \eqref{eq_i<k,j>l} and \eqref{eq_i<k,j<l} if $k\ne l$. Only the case $k=l=N$ remains, which is also okay.) Clearly, the $v_{jk}$ satisfy the determinant condition $D=\U$. (Indeed, one easily checks $s_N(D)=s_N(D^{NN})\U$. And $s_N(D^{NN})=\U$, because the generators of $\pol[N-1]$ fulfill the determinant condition.) Therefore $s_N$ is a well-defined algebra homomorphism. Clearly, $s_N$ is also a $*$-algebra homomorphism. (Indeed, the matrix $\tilde{V}:=[v_{kj}^*]_{jk}$ fulfills $\tilde{V}V=\U_N=V\tilde{V}$. Therefore, $\tilde{V}=V^*$.) Obviously, $s_N$ respects the counits. (Indeed, the two counits are multiplicative, and $s_N$ intertwines the right values on the generators. The same argument, though checking is slightly more involved, shows that $s_N$ also respects the the comultiplication; but we do not need that.)

\begin{cor}
The algebra $\pol[N-1]$ is canonically isomorphic to the quotient of the algebra $\pol$ by the extra relations $u_{kN}=\U\delta_{kN}=u_{Nk}$.
\end{cor}

\begin{proof}
The homomorphism $s_N$, clearly, respects the extra relations. Therefore it defines a homomorphism $\tilde{s}_N$ from the quotient of $\pol$ onto $\pol[N-1]$. Reading the definition backwards on the $(N-1)\times(N-1)$-submatrix, all of the relations in \eqref{eq_Rq} (for $\pol[N-1]$!) are fulfilled. Thanks to the extra relations, also the quantum determinant is sent to $\U$.
\end{proof}

Clearly, this $\tilde{s}_N$ is also a $*$-isomorphism and counits (and comultiplications) are the same. This makes applicable Lemma \ref{exrellem} when we wish to check if representations of $\pol$ and their cocycles live on the quantum subgroup $\pol[N-1]$.

\begin{cor} \label{lifecor}
Let $\pi$ ba a $*$-representation $\pi$ of $\pol$ and $\eta$ a $\pi$-$\e$-cocycle.
\begin{enumerate}
\item \label{lc1}
$\pi$ lives on $SU_q(N-1)$ if and only if $\pi(u_{kN})=\id_H\delta_{kN}=\pi(u_{Nk})$.

\item \label{lc2}
If $\pi$ lives on $SU_q(N-1)$, then $\eta$ lives on $SU_q(N-1)$ if and only if $\eta(u_{kN})=0=\eta(u_{Nk})$.
\end{enumerate}
\end{cor}

When, in Section \ref{decompSEC}, we also take into account operator theoretic statements, this corollary improves considerably. For all representations (see the beginning of Subsection \ref{repdSSEC}) and at least for (all) cocycles with respect to certain representations (see the proof of Corollary \ref{subcoccor}), it is sufficient to check only the respective condition regarding $u_{NN}$.

\medskip\noindent
Throughout, we also will need the iterated homomorphisms
\begin{equation} \label{snN}
s_{n,N}
~:=~
s_{n+1}\circ\ldots\circ s_N
~=~
\breve{t}_n\circ t_n\circ\ldots\circ\breve{t}_{N-1}\circ t_{N-1}
\hspace{10ex}
(n<N),
\end{equation}
which establish $SU_q(n)$ as a quantum subgroup of $SU_q(N)$. In Section \ref{sec_uqn}, we will also need $\breve{s}_N:= t_{N-1}\circ\breve{t}_N\colon U_q(N)\rightarrow U_q(N-1)$ and its iterates
\begin{equation} \label{BsnN}
\breve{s}_{n,N}
~:=~
\breve{s}_{n+1}\circ\ldots\circ\breve{s}_N
~=~
t_n\circ\breve{t}_{n+1}\circ\ldots\circ t_{N-1}\circ\breve{t}_N
\hspace{10ex}
(n<N),
\end{equation}
which establish $U_q(n)$ as a quantum subgroup of $U_q(N)$. Note that 
\begin{align} \label{sBstCR}
\breve{t}_n\circ\breve{s}_{n,N}
&
~=~
s_{n,N}\circ\breve{t}_N,
&
t_n\circ s_{n+1,N+1}
&
~=~
\breve{s}_{n,N}\circ t_N.
\end{align}

\medskip
We close this subsection by collecting some more relations. By definition, the generators $u_{jk}$ of $U_q(N)$ (and, therefore, also the generators of $\pol$) satisfy the basic commutation relations in \eqref{eq_Rq}. We will need frequently the following special cases.
\begin{subequations} \label{eq_NRq}
\begin{align}
\hspace{15ex}
u_{jN}u_{NN}&~=~ qu_{NN}u_{jN}, \label{eqN_i<k}\\
 u_{Nk}u_{NN}&~=~ qu_{NN}u_{Nk}, \label{eqN_j<l}\\
 u_{jN}u_{Nk}&~=~ u_{Nk}u_{jN}, \label{eqN_i<k,j>l}\\
 u_{jk}u_{NN} &~=~ u_{NN}u_{jk}-({\textstyle\frac{1}{q}}-q)u_{jN}u_{Nk}, \label{eqN_i<k,j<l}
\hspace{10ex}
\end{align}
\end{subequations}
for $j,k<N$. Recall that (as explained in brackets after Equation \eqref{eq_Du*}) the convolution is defined such that the unitarity conditions in \eqref{eq_U} are fulfilled. Additionally, one may verify the following commutation
{relations}
among the generators and their adjoints.
\begin{subequations}
\begin{eqnarray}
  u_{ij}u_{kl}^*&=& u_{kl}^*u_{ij} \quad \mbox{for} \; i\neq k, j\neq l,
\label{eq_*neq}\\
  u_{ij}u_{kj}^*&=& qu_{kj}^*u_{ij}-(1-q^2)\sum_{p<j} u_{ip}u_{kp}^* \quad
\mbox{for} \; i\neq k,\label{eq_*ineq}\\
  u_{ij}u_{il}^*&=& {\textstyle\frac{1}{q}}u_{il}^*u_{ij}+({\textstyle\frac{1}{q}}-q)\sum_{s>i} u_{sl}^*u_{sj}
\quad \mbox{for} \; j\neq l, \label{eq_*jneq}\\
  u_{ij}u_{ij}^*&=& u_{ij}^*u_{ij}+(1-q^2)\sum_{s>i}
u_{sj}^*u_{sj}-(1-q^2)\sum_{p<j} u_{ip}u_{ip}^*. \label{eq_*ij}
\end{eqnarray}
\end{subequations}
Here are some special consequences.
\begin{subequations}
\begin{eqnarray}
  \label{eq_*jN}
u_{Nj}u_{Nk}^* &=& {\textstyle\frac{1}{q}}u_{Nk}^*u_{Nj} \quad \mbox{for} \; j\neq k,\\
	\label{eq_*Nj}
u_{jN}u_{kN}^* &=& {\textstyle\frac{1}{q}} u_{kN}^*u_{jN} \quad \mbox{for} \; j\neq k,\\
  \label{eq_NN}
u_{NN}^*u_{NN} &=& q^2u_{NN}u_{NN}^*+(1-q^{2}) \one,
\end{eqnarray}
\end{subequations}
for $j,k<N$. Equation \eqref{eq_*jN} follows from \eqref{eq_*jneq}. Moreover, 
\eqref{eq_*ineq} together with the unitarity condition \eqref{eq_U} implies 
that 
\begin{eqnarray*} 
u_{jN}u_{kN}^*&=& qu_{kN}^*u_{jN}-(1-q^2)\sum_{p<N} u_{jp}u_{kp}^* 
=
qu_{kN}^*u_{jN}+(1-q^2)u_{jN}u_{kN}^*, 
\end{eqnarray*}
which shows \eqref{eq_*Nj}. Finally, we get \eqref{eq_NN} thanks to 
\eqref{eq_*ij}:
$$
u_{NN}u_{NN}^*= u_{NN}^*u_{NN}-(1-q^2)\sum_{p<N} u_{Np}u_{Np}^*
 = u_{NN}^*u_{NN}-(1-q^2)(\one- u_{NN}u_{NN}^*).
$$

Now, passing to the generators $u_{jk}$ of $\pol$, additionally, we have $D=D^{-1}=\U$. The involution simplifies to
\begin{equation} 
 \label{eq_u*}
 u_{jk}^* := (-q)^{k-j}D^{jk}.
\end{equation}
One may show that also the twisted determinant condition in \eqref{eq_TD} is satisfied.

\medskip
\subsection{A key lemma}
\label{keySSEC}

The following little lemma is key for the approximation results in Section \ref{decompSEC}. The idea for the approximation is taken from Sch{\"u}rmann and Skeide \cite{Ske94,schurmann+skeide98}. But the lemma would have simplified considerably also the proofs in \cite{Ske94,schurmann+skeide98}, which, in fact, will be reproved here. (For optical reasons, we write $\U$ instead of $\id_H$.)

\begin{lemma} \label{(1-a)lemma}
Let $a$ be a contraction on a Hilbert space $H$. Then the following are equivalent:
\begin{enumerate}
\item \label{(1-a).1}
$\overline{(\U-a)H}=H$.
\vspace{1ex}

\item \label{(1-a).2}
$\lim\limits_{p\uparrow 1}\frac{\U-a}{\U-pa}=\U$, strongly.
\vspace{.5ex}

\item \label{(1-a).3}
$\lim\limits_{p\uparrow 1}\frac{\U-a}{\U-pa}=\U$, weakly.
\vspace{.5ex}

\item \label{(1-a).4}
$\U-a$ is injective.
\end{enumerate}
Moreover, under any of the conditions, $\displaystyle\lim_{p\uparrow 1}\frac{1-p}{\U-pa}=0$, strongly.
\end{lemma}

\begin{proof}
Obviously, \eqref{(1-a).2}$\Rightarrow$\eqref{(1-a).3}. The approximation in \eqref{(1-a).2} or \eqref{(1-a).3} shows that every $x$ is in the closure of the range of $\U-a$ in the respective topology. So, clearly, \eqref{(1-a).2}$\Rightarrow$\eqref{(1-a).1} and \eqref{(1-a).3}$\Rightarrow\overline{(\U-a)H}^w=H$.  Since weak and norm closure of linear subspaces of Hilbert spaces coincide, we also get \eqref{(1-a).3}$\Rightarrow$\eqref{(1-a).1}.

For \eqref{(1-a).1}$\Rightarrow$\eqref{(1-a).2}, let us start with the observations that, for all $0\le p<1$,
\begin{align*}
\U-\frac{\U-a}{\U-pa}
&
~=~
a\frac{1-p}{\U-pa}
&&
~~~~~~\text{and}~~~~~~
&
\Bigl\|\frac{1-p}{\U-pa}\Bigr\|
&
~\le~
\frac{1-p}{1-p\|a\|}
~\le~
\frac{1-p}{1-p}
~=~
1.
\end{align*}
So, $\bigl\|\frac{\U-a}{\U-pa}\bigr\|\le2$ and, therefore, it suffices to check strong convergence on the total subset $(\U-a)H$ of $H$. So, let us choose $y\in H$ and check strong convergence on $x=(\U-a)y$. We find
$$
\Bigl\|x-\frac{\U-a}{\U-pa}x\Bigr\|
~=~
\Bigl\|a\frac{1-p}{\U-pa}(\U-a)y\Bigr\|
~\le~
1(1-p)2\|y\|
~\to~
0.
$$
This also shows the supplement $\frac{1-p}{\U-pa}\to0$, strongly.

So, we have closed \eqref{(1-a).1}$\Rightarrow$\eqref{(1-a).2}$\Rightarrow$\eqref{(1-a).3}$\Rightarrow$\eqref{(1-a).1}. To show equivalence with \eqref{(1-a).4}, it is enough to observe that, firstly, \eqref{(1-a).1} and \eqref{(1-a).4} are ``dual under adjoint'' to
{each other}
($a$ injective if and only if $a^*H$ is dense in $H$), and that, secondly, \eqref{(1-a).3} is invariant under taking adjoints.
\end{proof}

Let us note that the equivalence of \eqref{(1-a).1} and \eqref{(1-a).4} (the latter expressed in the form $\overline{(\U-a^*)H}=H$) is well-known. It can easily be shown directly; see, for instance, the elementary argument that occurs in the proof of Shalit \cite[Theorem 8.2.7]{Sha17} or the occurrence as a corollary of dilation theory in the survey
{by}
Levy and Shalit in \cite{LeSha11p}. But,
{we will neeed the implication}
\eqref{(1-a).4}$\Rightarrow$\eqref{(1-a).2}, and for that we do not know a source.

Actually the statement proved in the proof of \cite[Theorem 8.2.7]{Sha17} is slightly better than \eqref{(1-a).1}$\Leftrightarrow$\eqref{(1-a).4}. As we need it, we state it and also furnish a (different) proof.

\begin{lemma}
Under the same hypothesis: $ax=x$ $\Leftrightarrow$ $a^*x=x$ 
~(for all $x\in H$).
\end{lemma}

\begin{proof}
Of course, we are only interested in the case, when $x\ne0$ so that $x$ would be an eigenvector. But for the proof there is no difference. We have
\begin{align*}
\|ax-x\|^2
&
~=~
\,\,\,\,\,\,\|ax\|^2-
\langle x,a^*x\rangle-
\langle x,ax\rangle+
\|x\|^2,
\\
\|a^*x-x\|^2
&
~=~
\,\,\,\,
\|a^*x\|^2-
\langle x,a^*x\rangle-
\langle x,ax\rangle+
\|x\|^2.
\end{align*}
If the first row, is $0$, then the second row would be negative if $\|a^*x\|<\|ax\|$. We get
$$
\|x\|
~=~
\|ax\|
~\le~
\|a^*x\|
~\le~
\|a^*\|\,\|x\|
~=~
\|x\|.
$$
Therefore, $\|a^*x-x\|=\|ax-x\|=0$. The other direction follows by $a\leftrightarrow a^*$.\end{proof}

\medskip
\begin{cor} \label{pqcor}
For any contraction $a\in B(H)$, the Hilbert space $H$ decomposes uniquely into invariant subspaces $H_0$ and $H_1$ such that $a$ acts on $H_0$ as identity and such that $\U-a$ is injective on $H_1$. Moreover, $H_0=\ker(\U-a)$ and $\frac{\U-a}{\U-pa}$ converges strongly to the projection onto $H_1$.
\end{cor}

\newpage

\addtocontents{toc}{\vspace{1.5ex}}
\section{Classification of gaussian generating functionals} \label{sec_gaussian}

In this Section we investigate the gaussian generating functionals on $\pol$
and their Sch{\"u}rmann triples. We shall see that gaussian generating functionals on $\pol$ are classified by $(N-1)$ real 
numbers, which captures the freedom in the choice of a drift term, and a positive real $(N-1)\times (N-1)$-matrix , which captures the freedom in choosing a gaussian generating functional $\psi$ that satisfies $\psi\circ\cP=\psi$. Contrary to the case $N=2$, for $N\geq 3$ there exist gaussian pairs $(\id_H\e,\eta)$ that cannot be completed to a Sch{\"u}rmann triple. 

The first thing we have to do, also in order to actually indicate a projection $\cP$, is to find a  hermitian basis extension $E_1$ from $K_2$ to $K_1$. (See Subsection \ref{sec_S_trip}.) Our algebra is generated by the elements $u_{jk}-\U\delta_{jk}\in K_1$, their adjoints,  and $\U$. Since the elements of $E_1$ must be in $K_1$ but not in $K_2$, it is clear that we have to search them among the hermitian linear combinations of the $u_{jk}-\U\delta_{jk}$ and their adjoints. We put
$$
d_j
~:=~
\frac{u_{jj}-u_{jj}^*}{2i}.
$$

\begin{lemma} \label{lem_kernel_descr}
~
\begin{enumerate}
\item[(a)] $u_{jk}, u_{jk}^*\in K_2$
{ (actually, $u_{jk}, u_{jk}^*\in K_\infty$),  for $j\neq k$,}
\item[(b)] $(u_{jj}-\one)+(u_{jj}^*-\one) \in K_2$,
\item[(c)] $d_1+d_2+\ldots+d_N\in K_2$,
\item[(d)] $d_jd_k-d_kd_j\in K_3$ (actually, $d_jd_k-d_kd_j\in K_\infty$).
\end{enumerate}
\end{lemma}

\begin{proof}
(a) 
Uniting appropriately Relations \eqref{eq_i<k} and \eqref{eq_j<l}, we obtain $u_{jk}u_{ll}=qu_{ll}u_{jk}$ for $j\neq k$ and $l:=\max(j,k)$.
Therefore (expanding the brackets on the right-hand side),
$$
u_{jk}=\frac{q(u_{ll}-\one)u_{jk}-u_{jk}(u_{ll}-\one)}{1-q}.
$$
Since $u_{ll}-\one, u_{jk}\in K_1$, we get $u_{jk}\in K_2$. (By induction it follows that $u_{jk}\in K_n$ for all integers $n\ge 1$.)

(b) 
By the unitarity relation in \eqref{eq_U} we see that 
$$
\one-u_{jj}u_{jj}^*=\sum_{p\neq j} u_{jp} u_{jp}^*\in K_\infty.
$$
Hence $(u_{jj}-1)+(u_{jj}-1)^*=-(\one-u_{jj}u_{jj}^*)- (u_{jj}-1)(u_{jj}-1)^*
\in K_2$. 

(c) 
Putting $v_j:=u_{jj}-\one \in K_1$, we obtain
$$u_{11}\ldots u_{NN}=(v_1+\one)\ldots (v_N+\one)
=
 \one+ (v_1+\ldots+v_N)+ 
\text{terms in $K_2$}.
$$
Therefore,
$$
 v_1+\ldots+v_N + (\one-u_{11}\ldots u_{NN}) 
\in K_2,
$$
Since $D=\U$, we have
$$ 
\one - u_{11}\ldots u_{NN}=\sum_{\sigma\in S_N, \sigma \neq {\rm id}}
(-q)^{i(\sigma)} u_{1,\sigma(1)}\ldots u_{N,\sigma(N)}.
$$
Since for $\sigma\ne\id$ there is at least one $j$ with $j\ne\sigma(j)$, we see by Part (a) that the right-hand side is in $K_\infty$. So, $v_1+\ldots+v_N\in K_2$, hence,
$$
d_1+\ldots+d_N
 = 
\frac{(v_1+\ldots+v_N)-(v_1+\ldots + v_N)^*}{2i}
  \in 
K_2.
$$

(d)
This follows from \eqref{eq_i<k,j<l}, \eqref{eq_*neq}, and Part (a).\end{proof}

\lf
By (a), (b), and (c), we have:

\begin{cor}
Put $E_1:=\{d_2,\ldots,d_N\}$. Then the set $E_1\cup K_2$ spans $K_1$.
\end{cor}

\lf
By (d) and Corollary \ref{K3gcor}, we have:

\begin{cor} \label{cor_gaussian_existence}
The gaussian cocycle $\eta$ can be completed to a Sch{\"u}rmann triple $(\id_H\e,\eta,\psi)$ if and only if $\eta$ is hermitian.
\end{cor}

\lf
\begin{remark}
Recall that, in the proof of Corollary \ref{K3gcor}, we did show that the cocycles of gaussian generating functionals are hermitian under an extra condition (fulfilled by $\pol$, by the lemma). The backwards direction, still depends on Sch{\"u}rmann's \cite[Proposition 5.1.11]{schurmann93}. In the sequel, after completing the discussion of $E_1$ and $\cP$, we will construct (explicitly and in a classifying way) for each hermitian gaussian cocycles of $\pol$ a generating functional, thus, making Corollary \ref{K3gcor} (for $\pol$) independent of \cite[Proposition 5.1.11]{schurmann93}.
\end{remark}

\medskip
To show that $E_1$ is a basis extension, it remains to show that the elements $d_j$ of $E_1$ are linearly independent and are not in $K_2$. To that goal, let us consider the family of characters defined by
$$
\e_{\theta_2,\ldots, \theta_N}(u_{kl}):= e^{i\theta_k} \delta_{k,l},
$$
where $\theta_2,\ldots,\theta_N\in\mathbb{R}$, and where $\theta_1$ is determined by $\sum_{k=1}^N\theta_k=0$. (Of course, one easily verifies directly that this (well)defines a $*$-homomorphism into $\C$. But, see also Remark \ref{TNrem}.) One easily verifies that the functionals $\e'_j:=\frac{\partial\e_{\theta_2,\ldots, \theta_N}}{\partial\theta_j}\big|_{\theta_2=\ldots=\theta_N=0}$ $(j=2,\ldots,N)$ (pointwise derivative) vanish on $K_2$ and satisfy
$$
\e'_j(d_k)
~=~
\delta_{jk}.
$$

\begin{cor}
The elements $d_2,\ldots,d_N$ in $E_1$ are linearly independent and not in $K_2$. Therefore, by Corollary \ref{Gdelcor}, $E_1$ is a hermitian basis extension from $K_2$ to $K_1$.
\end{cor}

The $\e'_k$ here coincide, indeed, with the functionals $\e'_k$ occurring in Subsection \ref{ssec_gaussian} from the basis extension. We, now, also can fix our $\cP$ as in \eqref{P-form}. By Proposition \ref{Gcocprop}, we get all gaussian cocycles in the form
$$
\eta
~=~
\sum_{j=2}^N\eta_j\e'_j.
$$
But what are the hermitian ones, and how do they give rise to a generating functional?

Well, the first question is easy: $\eta$ is hermitian if and only if the matrix with entries $r_{jk}:=\langle\eta_j,\eta_k\rangle$ is real (hence, symmetric). For $N\ge3$, it is easy to write down gaussian cocycles that violate this. (See also \cite[Proposition 2.3]{dfks18}.)

\begin{cor} \label{ngccor}
For $N\ge3$, the quantum group $SU_q(N)$ does not have Property (GC).
\end{cor}

This contrasts the fact that $SU_q(2)$ has Property (AC); essentially, \cite[Lemma 2.6+Theorem 2.8]{Ske94} or \cite[Lemma 3.2+Theorem 3.3]{schurmann+skeide98}.

As for the second question: Almost as easy as for the one-dimensional case in  Proposition \ref{Gdelprop}, one checks by direct verification that, if the matrix $(r_{jk})$ is real then the functional
$$
\psi
~:=~
\frac{1}{2}\sum_{j,k=2}^Nr_{j,k}\e''_{j,k},
$$
with $\e''_{jk}:=\frac{\partial^2\e_{\theta_2,\ldots, \theta_N}}{\partial\theta_j\,\partial\theta_k}\big|_{\theta_2=\ldots=\theta_N=0}$ $(j,k=2,\ldots,N)$, has an $\eta$-induced $2$-coboundary (being, therefore, conditionally positive and $0$-normalized). Note  that $\e''_{jk}(d_l)=0$ as soon as $j\ne l$ or $k\ne l$. (In particular, it is $0$ if $j\ne k$.) And  $\e''_{jj}(d_j)=0$, by direct computation. So, $\psi$ also fulfills $\psi\circ\cP=\psi$ (being, therefore, hermitian).

Any real positive matrix $(r_{jk})$ may occur. Indeed, denote by $(q_{jk})$ the unique positive square root, which, necessarily, also has real entries. (Just diagonalize by a real unitary.) Then, the gaussian cocycle $\eta:=\sum_{j=2}^N\eta_j\e'_j$ with
{$\eta_j\in\C^{N-1}$}
having $\ell$-coordinate $q_{\ell j}$ ($\ell=2,\ldots,N$) has the matrix $\langle\eta_j,\eta_k\rangle=r_{jk}$.

Obviously, $\psi$ determines the $r_{jk}$. Adding also a drift term, we, thus, obtain:

\begin{thm} \label{Gclassthm}
Gaussian generating functionals on $\pol$ are parametrized one-to-one by $N-1$ real numbers $r_j$ $(j=2,\ldots,N)$ and
{ a positive real $(N-1)\times(N-1)$-matrices $(r_{jk})_{j,k=2,\ldots,N}$}
as
\begin{equation} \label{eq_gaussian_functional}
\psi
~=~
\sum_{j=2}^Nr_j\e'_j+\frac{1}{2}\sum_{j,k=2}^Nr_{j,k}\e''_{j,k}.
\end{equation}
\end{thm}

\medskip
We add that our projections $\cP$ for $\pol$ and for $\pol[N-1]$, are compatible with the subgroup structure:

\begin{prop} \label{subNPprop}
\hfill $\cP\circ s_N=s_N\circ\cP$. \hspace{40ex} {~}
\end{prop}

\begin{proof}
{The homomorphism}
$s_N$ (see \eqref{eq_suq_inclusion}) sends $d_N$ to $0$ and it sends the $d_n$ $(2\le n\le N-1)$ of $\pol$ to the $d_n$ of $\pol[N-1]$. Therefore, $s_N(E_1)=E^{SU_q(N-1)}_1\cup\{0\}$, and the statement follows from Corollary \ref{subPcor}.
\end{proof}

And iterating:

\begin{cor}
\hfill $\cP\circ s_{n,N}=s_{n,N}\circ\cP$. \hspace{37ex} {~}
\end{cor}

\medskip
\begin{remark} \label{TNrem}
Note that the classical $(N-1)$-torus $\mathbb{T}^{N-1}$ may be identified with the quantum group generated by $N$ commuting unitaries $u_j$ subject to the relation $u_1\ldots u_N=\U$. Sending $u_{jk}$ to $u_j\delta_{jk}$ defines a $*$-homomorphism $\tau_N$, identifying $\mathbb{T}^{N-1}$ as a quantum subgroup of $SU_q(N)$. Moreover, the family $\e_{\theta_2,\ldots, \theta_N}$ lives on $\mathbb{T}^{N-1}$. This shows several things:
\begin{enumerate}
\item \label{TN1}
The gaussian generating functionals of $\pol$ live on $\mathbb{T}^{N-1}$; in this sense, $SU_q(N)$ and $\mathbb{T}^{N-1}$ have the same gaussian generating functionals. Effectively, we could have deduced Theorem \ref{Gclassthm}, applying the results about classical compact Lie groups in \cite{skeide99} to $\mathbb{T}^{N-1}$. (Our approach here is simpler and improves also on \cite{skeide99}.) Obviously, the projection $\cP$ for $SU_q(N)$ lives on $\mathbb{T}^{N-1}$, too, and the (unique!) map $\tilde{\cP}$ that illustrates it, is the projection $\cP$ for $\mathbb{T}^{N-1}$.

\item \label{TN2}
Every quantum (semi)group $\cG$ sitting as $SU_q(N)\supset\cG\supset\mathbb{T}^{N-1}$, has the same gaussian parts. Moreover, the projection $\cP$ for $\cG$ may be chosen compatible with those for $SU_q(N)$ and $\mathbb{T}^{N-1}$.
\end{enumerate}
\end{remark}

\newpage

\addtocontents{toc}{\vspace{1.5ex}}
\section{Decomposition} \label{decompSEC}

This is the central Section of these notes. We decompose an arbitrary representation $\pi$ of $SU_q(N)$ into a unique direct sum $\pi=\pi_1\oplus\pi_2\oplus\ldots\oplus\pi_N$, where $\pi_n$ lives on $SU_q(n)$ for $2\le n\le N$, and where $\pi_1$
{ is the maximal gaussian part (living on the trivial quantum subgroup $SU_q(1)\cong\{e\}$, since $\pi_1$ is the trivial representation $\pi_1=\id_{H_1}\e$).} 
Then, on the completely non-gaussian part, we show that each cocycle $\eta$ is determined by the vectors $\eta(u_{nn})$ and can be approximated by coboundaries. Therefore, $SU_q(N)$ possesses Property (NC), hence, (LK). Since the cocycles $\eta_n$ with respect to $\pi_n$ obtained from the decomposition of $\pi$ are completely non-gaussian whenever $n\ge2$, we also get a decomposition $\psi=\psi_1+\ldots+\psi_N$ of $\psi$ into generating functionals $\psi_n$ that live on $SU_q(n)$.

\medskip
\subsection{Decomposition of representations} \label{repdSSEC}

Obviously, a representation $\rho$ of $SU_q(N)$ that lives on $SU_q(N-1)$ sends $u_{NN}$ to $\id$. But, this condition is also sufficient. Indeed, we have
\begin{equation} \label{ker-uNN}
\U-u_{NN}^*u_{NN}
~=~
\sum_{k=1}^{N-1}u_{kN}^*u_{kN}.
\end{equation}
So, if the left-hand side is sent by $\rho$ to $0$, then the sum of the positive operators to which $\rho$ sends the right-hand side, must be $0$ as well. So, $\rho(u_{kN})=0$ for all $1\le k\le N-1$. Starting from $\U-u_{NN}u_{NN}^*$, the same argument shows that $\rho(u_{Nk})=0$ for all $1\le k\le N-1$. Clearly, the remaining matrix $(\rho(u_{jk}))_{1\le j,k\le N-1}$ has to be unitary. By Corollary \ref{lifecor}, $\rho$ lives on $SU_q(N-1)$.

\begin{lemma}
Let $\pi$ be a representation of $\pol$. Then the subspace $\ker(\id-\pi(u_{NN}))$ is invariant for $\pi$.
\end{lemma}

\begin{proof}
Let $f$ be in $\ker(\id-\pi(u_{NN}))$. In other words, let $f$ be such that $\pi(u_{NN})f=f$.

Applying $\pi$ to Equation \eqref{ker-uNN}
and, then, the positive functional $\langle f,\bullet f\rangle$, we get $\pi(u_{kN})f=0$ for all $1\le k\le N-1$. Just as easy, from the other unitarity condition we get $\pi(u_{Nk}^*)f=0$ for all $1\le k\le N-1$. From Relation \eqref{eq_*ij}, we get
$$
u_{Nk}^*u_{Nk}
~=~
u_{Nk}u_{Nk}^*
+(1-q^2)\sum_{j<k}u_{Nj}u_{Nj}^*,
$$
so that also $\pi(u_{Nk})f=0$ for all $1\le k\le N-1$.
So, $\ker(\id-\pi(u_{NN}))$ is invariant for all $\pi(u_{kN})$ and $\pi(u_{Nk})$.

From Equation \eqref{eq_i<k}, for $j,k\le N-1$ we get $\pi(u_{NN})(\pi(u_{jk})f)=\frac{\pi(u_{jk})\pi(u_{NN})f}{q}=\frac{(\pi(u_{jk})f)}{q}$. Since $\pi(u_{NN})$ is a contraction and $\frac{1}{q}>1$, we get also $\pi(u_{jk})f=0$.

So,  $\ker(\id-\pi(u_{NN}))$ is invariant under all $\pi(u_{jk})$ ($j,k\le N$) and since, by Equation \eqref{eq_u*}, the $u_{jk}$ generate the $*$-algebra $\pol$ as an algebra, $\ker(\id-\pi(u_{NN}))$ is invariant for $\pi$.\end{proof}

\begin{cor}
Putting $H_N:=\ker(\id-\pi(u_{NN}))^\perp$, the representation $\pi$ decomposes uniquely into a part $\pi_N$ acting on $H_N$ where $\pi_N(\U-u_{NN})$ is injective, and a part on $H_N^\perp$ that lives on $SU_q(N-1)$.
\end{cor}

And by induction:

\begin{thm} \label{repdecthm}
Let $\pi$ be a representation of $\pol$ on a Hilbert space $H$. Then $\pi$ decomposes uniquely into representations $\pi_n$ on invariant subspaces $H_n$ $(n=1,\ldots,N)$ such that
\begin{itemize}
\item
$\pi_1$ is gaussian.

\item
$\pi_n$ $(2\le n\le N)$ lives on $SU_q(n)$.
\end{itemize}

Moreover, each $\pi_n(\U-u_{nn})$ is injective.
\end{thm}

\begin{proof}
For $N=2$, the statement follows directly from the lemma. And the corollary provides the inductive step.
\end{proof}
Note that $\pi_1(\U-u_{11})$ is injective if and only if $\pi_1=0$.

\medskip
\subsection{Decomposition of cocycles} \label{cocdSSEC}

Well, we know that if a $*$-representation $\pi$ decomposes into representations $\pi_n$, then a $\pi$-$\e$-cocycle $\eta$ decomposes into $\pi_n$-$\e$-cocycle $\eta_n$. We now convince ourselves that not only (in the decomposition from the preceding subsection) $\pi_n$ lives on $SU_q(n)$, but
{that, for $n=2,\ldots,N$,}
also $\eta_n$ lives on $SU_q(n)$ and is determined by $\eta_n(u_{nn})$.

Let us start with some general cocycle computations.

\begin{lemma} \label{lem_eta_relations}
Let $\pi$ be a $*$-representation of $\pol$ and let $\eta$ be a $\pi$-$\e$-cocycle. Then   
\begin{subequations} \label{eq_eta}
\begin{eqnarray}
\label{eq_eta_jn}
 \eta(u_{jN}) &=& \textstyle\frac{-1}{\id-q\pi(u_{NN}))}\pi(u_{jN})\eta(u_{NN}), 
\\
\label{eq_eta_nk}
 \eta(u_{Nk}) &=& \textstyle\frac{-1}{\id-q\pi(u_{NN}))}\pi(u_{Nk})\eta(u_{NN}), 
\end{eqnarray}
\end{subequations}
and
\begin{equation}\label{eq_eta_jk}
 \pi(u_{NN}-\one)\eta(u_{jk}) = \Bigl(\pi( u_{jk}-\U\delta_{jk}) - \textstyle\frac{{\frac{1}{q}}-q}{\pi(\one-q^2 u_{NN})}\pi(u_{jN} u_{Nk})\Bigr)\, \eta(u_{NN})
\end{equation}
for any $j,k<N$.
\end{lemma}

\begin{proof}
If $a=u_{jN}$ or $a=u_{Nk}$
{ for $j,k<N$},
then $a\in \ker \e$ and 
$au_{NN}=qu_{NN}a$. Hence the cocycle property leads to
$$ \pi(a)\eta(u_{NN})+\eta(a)=q\pi(u_{NN})\eta(a).$$
Since $\pi(u_{NN})$ is a contraction, we know that
$\id-q\pi(u_{NN})$ is (boundedly) invertible. Thus,
$$
\eta(a) = \textstyle\frac{-1}{\id-q\pi(u_{NN}))}\pi(a)\eta(u_{NN}) 
$$
for any such element $a$.

On the other hand, if $a=u_{jk}$ with $j,k<N$, then 
the cocycle property applied to \eqref{eq_i<k,j<l} reads
\begin{eqnarray*}
\pi(u_{jk})\eta(u_{NN})+\eta(u_{jk}) &=& \eta(u_{jk}u_{NN}) =
\eta(u_{NN}u_{jk})-({\textstyle\frac{1}{q}}-q)\eta(u_{jN}u_{Nk}) \\
&=&
\pi(u_{NN})\eta(u_{jk})+\eta(u_{NN})\e(u_{jk})-({\textstyle\frac{1}{q}}-q)\pi(u_{jN})
\eta(u_{Nk}),
\end{eqnarray*}
so,
\begin{eqnarray*}
\lefteqn{ \pi(u_{NN}-\one)\eta(u_{jk})
= \pi(u_{jk}-\delta_{jk}\U) \eta(u_{NN}) + ({\textstyle\frac{1}{q}}-q)\pi(u_{jN})
\eta(u_{Nk}) }\\
&=& \big[\pi(u_{jk}-\delta_{jk}\U) - ({\textstyle\frac{1}{q}}-q)\pi(u_{jN})
\textstyle\frac{1}{\id-q\pi(u_{NN}))}\pi(u_{Nk})\big] \,\eta(u_{NN}) \\
&=& \big[ \pi(u_{jk}-\delta_{jk}\one) - ({\textstyle\frac{1}{q}}-q)\textstyle\frac{1}{\id-q^2\pi(u_{NN}))}\pi(u_{jN})
u_{Nk}\big] \,\eta(u_{NN}).
\end{eqnarray*}

\vspace{-5ex}
\end{proof}

\begin{cor} \label{uNNdetcor}
If $\pi(\U-u_{NN})$ is injective, then any  $\pi$-$\e$-cocycle $\eta$ is determined by its value $\eta(u_{NN})$.
\end{cor}

\begin{proof}
If $\pi(\U-u_{NN})$ is injective, then also $\eta(u_{jk})$ $(j,k<N)$ is determined by the lemma.
\end{proof}

\lf
{
\begin{lemma} \label{subcoclem}
Suppose for $n<N$ we have a $*$-representation $\pi$ of $\pol$ that lives on $\pol[n]$ and such that $\pi(\U-u_{nn})$ is injective. Then every $\pi$-$\e$-cocycle $\eta$ satisfies $\eta(u_{mm})=0$ for all $m>n$.
\end{lemma}

\begin{proof}
Applying $\eta$ to \eqref{eq_i<k,j<l} for $i=j=n$ and $k=l=m$, we obtain
$$
\pi(u_{nn})\eta(u_{mm})+\eta(u_{nn})
~=~
\pi(u_{mm})\eta(u_{nn})+\eta(u_{mm})-({\textstyle\frac{1}{q}-q})\pi(u_{nm})\eta(u_{mn}).
$$
$\pi$ lives on $\pol[n]$, so, $\pi(u_{mm})=\id$ and $\pi(u_{nm})=0$. Hence, the equation simplifies to
$$
\pi(u_{nn})\eta(u_{mm})
~=~
\eta(u_{mm}).
$$
Since $\pi(\U-u_{nn})$ is injective, we get $\eta(u_{mm})=0$.\end{proof}

\lf
\begin{cor} \label{subcoccor}
Such $\eta$ lives on $\pol[n]$, too.
\end{cor}
 
\begin{proof}
Since $\eta(U_{NN})=0$, by Equations \eqref{eq_eta} we get $\eta(u_{kN})=0=\eta(u_{Nk})$ for all $1\le k\le N$. By Corollary \ref{lifecor}, $\eta$ lives on $\pol[N-1]$. 

The result follows, now, by induction. (Applying the same argument to the representation $\tilde{\pi}_{N-1}$ and the cocycle $\tilde{\eta}_{N-1}$ on $\pol[N-1]$, we get that they live on $\pol[N-2]$, and so forth.)
\end{proof}

\lf
Consequently:

\begin{thm} \label{cocdecthm}
In the notations of Theorem \ref{repdecthm}: Every $\pi$-$\e$-cocycle $\eta$ decomposes (uniquely) into the direct sum over $\pi_n$-$\e$-cocycle $\eta_n$, where $\eta_1$ is gaussian and where $\eta_n$ lives on $\pol[n]$ for $2\le n\le N$. Moreover, $\eta$ is determined by its values $\eta(u_{nn})$ $(1\le n\le N)$.
\end{thm}

\begin{proof}
Only the last statement needs a proof. It follows from
$$
\eta(u_{nn})
~=~
0\oplus\ldots\oplus 0\oplus\eta_n(u_{nn})\oplus\ldots\oplus\eta_N(u_{nn}),
$$
because, by Corollary \ref{uNNdetcor}, each $\eta_{n'}(u_{nn})$ $(n'\ge n)$ is determined by $\eta_{n'}(u_{n'n'})$.\end{proof}
}

\lf
\begin{remark}
In the following subsection we will show that, for $2\le n\le N$, the cocycles $\eta_n$
are
limits of coboundaries.
So, it is legitimate to ask, why, in order to show that $\eta_n$ lives on $\pol[n]$, we did not appeal to Corollary \ref{subapprcor}. The point is that our basic approximation result, Proposition \ref{cocaprop}, cannot be applied directly to the cocycle $\eta_n$ on $\pol$, but only to a cocycle on $\pol[n]$.
Only after having shown in the present subsection that $\eta_n$ lives on $\pol[n]$, we have granted the cocycle $\tilde{\eta}_n$ on $\pol[n]$ such that
$\eta_n=\tilde{\eta}_n\circ s_{n,N}$
to which Proposition \ref{cocaprop} can be applied. (See the proof of Theorem \ref{cocapprthm}.)
\end{remark}

\medskip
\subsection{Approximation of cocycles} \label{cocaSSEC}

We now show that each of the cocycles $\eta_n$ $(n\ge2)$ in Theorem \ref{cocdecthm} can be approximated by coboundaries. (For a gaussian cocycle this is, obviously, not so, because gaussian coboundaries are $0$.)

\begin{prop} \label{cocaprop}
Let $\pi$ be a $*$-representation of $\pol$ such that $\pi(\U-u_{NN})$ is injective, and let $\eta$ be a $\pi$-$\e$-cocycle. Then the coboundaries
\begin{eqnarray} \label{cocadefi}
a
~\longmapsto~
\pi\circ(\id-\U\e)(a)\frac{-1}{\pi(\U-pu_{NN})}\eta(u_{NN})
~~~~~~
(0<p<1)
\end{eqnarray}
converge (pointwise on $\pol$) to $\eta$ for $p\to1$.
\end{prop}

\begin{proof}
This is essentially a consequence of Lemmata \ref{(1-a)lemma} and \ref{lem_eta_relations}. By the cocycle property, it suffices to control convergence on the generators $u_{jk}$.

For $a=u_{NN}$, we get $\lim_{p\to 1}\frac{\pi(\U-u_{NN})}{\pi(\U-pu_{NN})}\eta(u_{NN})=\eta(u_{NN})$.

For $a=u_{kN}$ or $a=u_{kN}$ $(k<N)$, by  Lemmata \ref{lem_eta_relations} and \ref{(1-a)lemma} we get
\begin{multline*}
\textstyle
\eta(a)
~=~
\frac{-1}{\pi(\U-qu_{NN})}\pi(a)\eta(u_{NN})
~=~
\lim_{p\to 1}\frac{-1}{\pi(\U-qu_{NN})}\pi(a)\frac{\pi(\U-u_{NN})}{\pi(\U-pu_{NN})}\eta(u_{NN})
\\[1ex]
\textstyle
~=~
\lim_{p\to 1}\pi(a)\frac{-1}{\pi(\U-pu_{NN})}\eta(u_{NN}).
\end{multline*}
Since $a\in\ker\e$, we have $\pi(a)=\pi\circ(\id-\U\e)(a)$.

After these two easy cases, here is the difficult one: For $a=u_{jk}$ $(j,k<N)$, by  Lemmata \ref{lem_eta_relations} and \ref{(1-a)lemma} we see that $\frac{-1}{\pi(\U-pu_{NN})}$ $\times$ the right-hand side of \eqref{eq_eta_jk} converges to $\eta(a)$. We are done if we show that the difference with the right-hand side of \eqref{cocadefi}, that is,
$$
\textstyle
\Bigl(
 \frac{-1}{\pi(\U-pu_{NN})}\big( \pi( u_{jk}-\U\delta_{jk}) - \textstyle\frac{{\frac{1}{q}}-q}{\pi(\one-q^2 u_{NN})}\pi(u_{jN} u_{Nk})\big) 
-
\pi(u_{jk}-\U\delta_{jk})\frac{-1}{\pi(\U-pu_{NN})}
\Bigr)
\eta(u_{NN})
\lf
$$
converges to $0$. We observe that the terms with $\delta_{jk}$ cancel out. Omitting $\eta(u_{NN})$, taking also into account that $\frac{1}{\pi(\U-pu_{NN})}$ and $\frac{1}{\pi(\U-q^2u_{NN})}$ commute, we remain with
\lf
$$
\textstyle
\pi(u_{jk})\frac{1}{\pi(\U-pu_{NN})}
-
\frac{1}{\pi(\U-pu_{NN})}\pi(u_{jk})
+
\frac{{\frac{1}{q}}-q}{\pi(\one-q^2 u_{NN})} \frac{1}{\pi(\U-pu_{NN})} \pi(u_{jN} u_{Nk}).
$$

\noindent
We will conclude the proof by showing that this converges to $0$, strongly. Recall that $q$ is fixed, and that multiplying the whole thing from the left with $\pi(\U-q^2u_{NN})$ to make disappear $ \frac{1}{\pi(\U-q^2u_{NN})}$ does not change
{whether this expression converges for $p\to1$ or not},
nor does it, in the case of convergence, change the answer to the question if the limit is $0$ or not. We get
\lf
\begin{eqnarray} \label{u_ij_diff}
\textstyle
\pi(\U-q^2u_{NN})\bigl[\pi(u_{jk}),\frac{1}{\pi(\U-pu_{NN})}\bigr]
+
\frac{{\frac{1}{q}}-q}{\pi(\U-pu_{NN})} \pi(u_{jN} u_{Nk}).
\end{eqnarray}

For simplicity, in the following computations (for fixed $p<1$) we omit the representation $\pi$. (In $\pol$ this does not make sense. But in the enveloping $C^*$-algebra it does, and after reinserting $\pi$ the result is the right one, because $\pi$ is bounded.) Let us first compute the commutator $\bigl[u_{jk},\frac{1}{\U-pu_{NN}}\bigr]$. From \eqref{eq_i<k,j<l}, we get by induction that
$$
[u_{jk},u_{NN}^s]
~=~
-({\textstyle\frac{1}{q}-q})(1+\ldots+q^{2(s-1)})u_{NN}^{s-1}u_{jN}u_{Nk},
$$
where the sum $1+\ldots+q^{2(s-1)}=\sum_{t=0}^{s-1}q^{2t}$ is understood to have $s$ summands (also if $s=0$); it coincides with the well known $q^2$-number $[s]_{q^2}$. We get
$$
\bigl[u_{jk},\frac{1}{\U-pu_{NN}}\bigr]
~=~
-({\textstyle\frac{1}{q}-q})\sum_{s=1}^\infty\sum_{t=0}^{s-1}q^{2t}p^su_{NN}^{s-1}u_{jN}u_{Nk}.
$$
Inserting this in \eqref{u_ij_diff}, expanding also $\frac{1}{\U-pu_{NN}}$, and omitting the common factor $(\frac{1}{q}-q)u_{jN}u_{Nk}$,we obtain
\begin{multline*}
(q^2u_{NN}-\U)\sum_{s=1}^\infty\sum_{t=0}^{s-1}q^{2t}p^su_{NN}^{s-1}+\sum_{s=0}^\infty(pu_{NN})^s
\\
~=~
\sum_{s=1}^\infty\sum_{t=0}^{s-1}q^{2(t+1)}p^su_{NN}^s-\sum_{s=1}^\infty\sum_{t=0}^{s-1}q^{2t}p^su_{NN}^{s-1}+\U+\sum_{s=1}^\infty(pu_{NN})^s
\\
~=~
\U+\sum_{s=1}^\infty\sum_{t=0}^sq^{2t}(pu_{NN})^s-p\sum_{s=1}^\infty\sum_{t=0}^{s-1}q^{2t}(pu_{NN})^{s-1}
\\
~=~
\sum_{s=0}^\infty\sum_{t=0}^sq^{2t}(pu_{NN})^s-p\sum_{s=0}^\infty\sum_{t=0}^sq^{2t}(pu_{NN})^s
~=~
(1-p)\sum_{s=0}^\infty\sum_{t=0}^sq^{2t}(pu_{NN})^s.
\end{multline*}
Reordering for powers of $q^2$, we get
$$
(1-p)\sum_{t=0}^\infty q^{2t}\sum_{s=t}^\infty(pu_{NN})^s
~=~
(1-p)\sum_{t=0}^\infty q^{2t}(pu_{NN})^t\sum_{s=0}^\infty(pu_{NN})^s
~=~
\textstyle
\frac{1}{\U-q^2pu_{NN}}\,\frac{1-p}{\U-pu_{NN}}.
$$
For $p\uparrow1$, the first factor converges in norm to $\frac{1}{\U-q^2u_{NN}}$. Under $\pi$, by the supplement of Lemma \ref{(1-a)lemma}, the second factor converges to $0$, strongly.\end{proof}

\begin{thm} \label{cocapprthm}
In the notations of Theorems \ref{repdecthm} and \ref{cocdecthm}: The cocycles $\eta_n$ $(2\le n\le N)$ are (pointwise) limits
$$
\eta_n
~=~
\lim_{p\uparrow 1}\pi_n\circ(\id-\U\e)\frac{-1}{\pi_n(\U-pu_{nn})}\eta_n(u_{nn})
$$
of coboundaries.
\end{thm}

\begin{proof}
Recall that by Theorem \ref{cocdecthm}, $\eta_n$ lives on $\pol[n]$. The result follows by applying Proposition \ref{cocaprop} to the (unique) cocycle $\tilde{\eta}_n$ on $\pol[n]$ such that $\eta_n=\tilde{\eta}_n\circ s_{n,N}$.
\end{proof}

\medskip
\subsection{Decomposition of generating functionals} \label{GFdSSEC}

We are now ready for the punch line.

\begin{prop} \label{GFcocprop}
Let $\pi$ be a $*$-representation of $\pol$ such that $\pi(\U-u_{NN})$ is injective. Then, every $\pi$-$\e$ cocycle $\eta$ determines a unique generating functional $\psi$ that completes $(\pi,\eta)$ to a Sch{\"u}rmann triple and satisfies $\psi\circ\cP=\psi$.
\end{prop}

\begin{proof}
This is a corollary of Proposition \ref{cocaprop} and Lemma \ref{cocaGFlem}.\end{proof}

\begin{cor} \label{GFcoccor}
In the notations of Theorems \ref{repdecthm} and \ref{cocdecthm}: For every $2\le n\le N$, the cocycle $\eta_n$ determines a unique generating functional $\psi_n$ that completes $(\pi_n,\eta_n)$ to a Sch{\"u}rmann triple and satisfies $\psi_n\circ\cP=\psi_n$. Moreover, $\psi_n$ lives on $SU_q(n)$.
\end{cor}

\begin{proof}
Let $\tilde{\psi}_n$ denote the generating functional granted (for some $\cP_n$ on $\pol[n]$) by Proposition \ref{GFcocprop} for the cocycle $\tilde{\eta}_n$ on $\pol[n]$. Then $\psi_n:=\tilde{\psi}_n\circ s_{n,N}\circ\cP$ does the job. By Proposition \ref{subNPprop} and its corollary, $\psi_n=\tilde{\psi}_n\circ s_{n,N}$.\end{proof}

\begin{cor} \label{NCcor}
$\pol$ has property {\normalfont (NC)}, hence, {\normalfont (LK)}.
\end{cor}

\begin{proof}
A cocycle $\eta$ being completely non-gaussian, means precisely that $\eta_1$ is $0$. Then $\psi:=\psi_2+\ldots+\psi_N$ from the preceding corollary, does the job.
\end{proof}

The $\psi_n$ are determined uniquely by $\psi_n\circ\cP=\psi_n$ and the requirement that $\tilde{\psi}_n$ has a Sch{\"u}rmann triple where $\tilde{\pi}_n(\U-u_{nn})$ is injective. We wish to capture this property without explicit reference to the Sch{\"u}rmann triple.

\begin{defi}
A generating functional $\psi$ on $\pol$ ($N\ge2$) is \hl{irreductible} if it is completely non-gaussian and if for every generating functional $\tilde{\psi}$ on $\pol[N-1]$, the statement $\psi-\tilde{\psi}\circ s_N$ is a generating functional implies the statement $\tilde{\psi}\circ s_N$ is a drift.
\end{defi}

This definition does the job:

\begin{prop}
Let $\psi$ be a generating functional on $\pol$ ($N\ge2$) and $(\pi,\eta,\psi)$ its Schürmann triple. Then, $\psi$ is irreductible if and only if, in Theorem \ref{repdecthm}'s decomposition, $\pi=\pi_N$.
\end{prop}

\begin{proof}
Obviously, the statement is true if $\psi$ is not completely non-gaussian. So, we assume that $\psi$ is completely non-gaussian.

We know from Theorem \ref{repdecthm} that $\pi=\pi_N\oplus\tilde{\pi}\circ s_N$ for some representation $\tilde{\pi}$ of $\pol[N-1]$. Moreover, $\tilde{\pi}_1=0$, because, otherwise, $\tilde{\pi}_1\circ s_N$ would contribute to $\pi_1$. So, if $N=2$, then there is nothing left to prove, so we assume $N\ge3$.

By Corollary \ref{subcoccor}, the cocycle $\eta-\eta_N$ with respect to $\tilde{\pi}\circ s_N$ also lives on $\pol[N-1]$, hence, has the form $\tilde{\eta}\circ s_N$ for a (unique) cocycle $\tilde{\eta}$ with respect to $\tilde{\pi}$.

By Corollary \ref{NCcor}, both pairs $(\pi_N,\eta_N)$ and $(\tilde{\pi},\tilde{\eta})$ may be completed to Schürmann triples $(\pi_N,\eta_N,\psi_N)$ and $(\tilde{\pi},\tilde{\eta},\tilde{\psi})$. Clearly, we may arrange $\psi_N$ such that $\psi=\psi_N+\tilde{\psi}\circ s_N$.

Clearly, $\tilde{\psi}\circ s_N$ is a drift if and only if $\tilde{\eta}\circ s_N=0$. Of course, if $\tilde{\pi}\circ s_N=0$, that is, if $\pi=\pi_N$, then $\tilde{\eta}\circ s_N=0$ so that $\tilde{\psi}\circ s_N$ is a drift. Conversely, since $\eta$ is cyclic, if $\tilde{\pi}\circ s_N\ne0$, that is, if $\pi\ne\pi_N$, then $\tilde{\eta}\circ s_N\ne0$ so that $\tilde{\psi}\circ s_N$ is a not drift.\end{proof}

As an immediate corollary, we get the main result:

\begin{thm} \label{psidecthm}
Let $\psi$ be a generating functional on $\pol$. Then $\psi$ decomposes uniquely into a sum $\psi:=\psi_1+\ldots+\psi_N$ such that the $\psi_n$ satisfy:
\begin{itemize}
\item
$\psi_1$ is gaussian.

\item
$\psi_n=\tilde{\psi}_n\circ s_{n,N}$ $(n\ge2)$ for some irreductible generating functional $\tilde{\psi}_n$ on $\pol[n]$ satisfying $\tilde{\psi}_n\circ\cP=\tilde{\psi}_n$.
\end{itemize}
\end{thm}

\noindent
Recall that, by definition, all $\psi_n$ $(2\le n)$, hence their sum, are completely non-gaussian.

\begin{cor}
$\psi=\psi_G+\psi_L$ with $\psi_G:=\psi_1$ and $\psi_L:=\psi_2+\ldots+\psi_N$ is the unique L{\'e}vy-Khintchine decompositions satisfying $\psi_L\circ\cP=\psi_L$.
\end{cor}

\lf\lf
\begin{remark}
~

\begin{enumerate}
\item
Recall from Section \ref{sec_gaussian} that $\pol$, unlike $\pol[2]$, does not have property {\normalfont (GC)} as soon as $N\ge3$.

\item
Like for $\pol[2]$, completely non-gaussian cocycles on $\pol$ are all limits of coboundaries and determined by $N-1$ vectors. However, unlike for $\pol[2]$, not all vectors in the representation spaces of $\pi_n$ may occur. (This is subject of Section \ref{sec_class}.) While for $\pol[2]$ the cocycle would be defined as the limit in Proposition \ref{cocaprop} (including a proof that the limit for $N=2$ exists for whatever vector we chose), for $\pol$, the fact that we start with a given cocycle was used essentially in the proof that the limit for $a=u_{jk}$ $(j,k<N)$ exists and gives the right result. (In fact, for the vector in the counter example in Proposition \ref{prop_counterexample}, the limit {\bf cannot} exist for all $a$, because, otherwise, it would define a cocycle.)
\end{enumerate}
\end{remark}

\newpage

\addtocontents{toc}{\vspace{1.5ex}}
\section{Parametrization of generating functionals} \label{sec_class}

In Section \ref{decompSEC}, we have shown that any Sch{\"u}rmann triple on $\pol$ can be 
decomposed (uniquely, in the sense explained in Theorem \ref{psidecthm}) into a gaussian part and a sum of generating functionals that live on $\pol[n]$ $(2\le n\le N)$ and which, as functionals on $\pol[n]$, are irreductible. In Section \ref{sec_gaussian}, we have classified the gaussian generating functionals. In this
section we also parametrize the irreductible generating functionals. Putting everything together by means of Theorem \ref{psidecthm}, we obtain a parametrization of all generating functionals on $SU_q(N)$, and, thus, up to quantum stochastic equivalence (up to unitary equivalence), of all (cyclic) L\'evy processes on this quantum group. 

\medskip
The irreductible generating functionals on $\pol$ are exactly those generating functionals $\psi$ that admit a Sch{\"u}rmann triple $(\pi,\eta,\psi)$ where $\pi(\U-u_{NN})$ is injective. Since, by Proposition \ref{GFcocprop}, every cocycle in such a triple determines the values of a generating functional on $K_2$, we may classify the irreductible generating functionals (up to unitary equivalence) by cyclic cocycles with respect to $*$-representations $\pi$ with injective $\pi(\U-u_{NN})$. However, concentrating on the cocycle and only, then, determining its representation (and check whether it fulfills the condition) is not really very practicable. Not for nothing, the order in Procedure \ref{proc} is to start with the representation $\pi$ and, then, to determine all its cocycles. Not for nothing,
{did we explain}
in Remark \ref{cycrem} that we do not usually require that the cocycle be cyclic. (This can be done better under symmetry conditions on $\psi$; see Das, Franz, Kula, and Skalski \cite{dfks15}.)

On the other hand, by Corollary \ref{uNNdetcor} (or, better, by Lemma \ref{lem_eta_relations}) we know that a cocycle $\eta$ with respect to 
{$\pi$ with injective $\pi(1-u_{NN})$,}
is determined by its value $\eta(u_{NN})$. It is, therefore, tempting to parametrize such cocycles (and, then, the functionals they determine) by vectors $\eta_{NN}\in H$ such that $\eta(u_{NN})=\eta_{NN}$. For $N=2$ (where the irreductible generating functionals are exactly the completely non-gaussian ones), we know from \cite{Ske94,schurmann+skeide98} that every vector in $H$ may occur as $\eta(u_{22})$ for a cocycle. However, for $N\ge3$, this is not so.

\begin{prop} \label{prop_counterexample}
For every $N\ge3$ there exists a $*$-representation $\pi$ of $\pol$ on $H$ with injective $\pi(\U-u_{NN})$ and a vector $\eta_{NN}\in H$ such that no $\pi$-$\e$-cocycle $\eta$ fulfills $\eta(u_{NN})=\eta_{NN}$.
\end{prop}

\begin{proof}
For the $N\times N$-matrix $(u_{jk})$ of $SU_q(N)$, let us refer by the $[m,n]$-block ($1\le m<n\le N$) to the submatrix with indices $j,k\in\{m,\ldots,n\}$. So far, we always embedded $SU_q(n)$ by, roughly speaking, ``identifying'' its defining $n\times n$-matrix with the $[1,n]$-block of $SU_q(N)$. (See the definition of $s_N$ in \eqref{eq_suq_inclusion} for the precise meaning of this in the case $n=N-1$.) It is noteworthy, that we may embed $SU_q(n)$ in the same way to any other $[m+1,m+n]$-block of $SU_q(N)$. Here, we are interested in particular in the
lower right
$n\times n$-square, that is, in the $[N-n+1,N]$-block of $SU_q(N)$. Since a $*$-representation $\pi_3$ of $\pol[3]$ gives, when embedded into the $[N-2,N]$-block of $\pol$, rise to a $*$-representation $\pi$ of $\pol$, and since $\pi$ has injective $\pi(\U-u_{NN})$ if and only if $\pi_3$ has injective $\pi_3(\U-u_{33})$, we see that it is enough to show the statement for $N=3$, only.

Now, for $N=3$, every $*$-representation $\rho$ of $\pol[2]$ gives rise to a $*$-representation $\rho_1$ of $\pol[3]$ when embedding $SU_q(2)$ into the $[1,2]$-block of $SU_q(3)$ and a $*$-representation $\rho_2$ of $\pol[3]$ when embedding $SU_q(2)$ into the $[2,3]$-block of $SU_q(3)$. Recall that the generators $u_{jk}$ of $\pol[2]$ are commonly written as
$$
\left(\begin{matrix}
u_{11}&u_{12}\\u_{21}&u_{22}
\end{matrix}\right)
~=~
\left(\begin{matrix}
\alpha&-q\gamma^*\\\gamma&\alpha^*
\end{matrix}\right).
$$
For $\rho$ let us choose the irreducible $*$-representation on the Hilbert space $H$ with ONB $(e_k)_{k\in\N_0}$ defined by
\begin{align*}
\rho(\alpha)
&
\colon
e_k\mapsto e_{k-1}\sqrt{1-q^{2k}},
&
\rho(\gamma)
&
\colon
e_k\mapsto e_kq
\end{align*}
(where $e_{-1}:=0$). We put $\pi:=\rho_1\star\rho_2$, so that
\begin{multline*}
\pi(U)
~=~
(\rho_1\otimes\rho_2)\Bigl(\sum_{i=1}^3u_{ji}\otimes u_{ik}\Bigr)
~=~
(\rho_1\otimes\rho_2)
\left( \left(\begin{array}{ccc} 
\alpha & -q\gamma^* & 0 \\
\gamma & \alpha^* & 0 \\
0 & 0 & \one 
\end{array}\right)
\otimes \left( \begin{array}{ccc} 
\one & 0 & 0 \\
0 & \alpha & -q\gamma^*\\
0 & \gamma & \alpha^* \\
\end{array}\right)\right)
\\
~=~
 \left( \begin{array}{ccc} 
\rho(\alpha) \otimes \id & -q\rho(\gamma)^* \otimes \rho(\alpha) & q^2 \rho(\gamma)^* \otimes  \rho(\gamma)^* 
\\
\rho(\gamma) \otimes  \id & \rho(\alpha)^* \otimes  \rho(\alpha) & -q \rho(\alpha)^* \otimes \rho(\gamma)^* 
\\
0 & \id \otimes \rho(\gamma) & \id \otimes \rho(\alpha)^*
\end{array}\right).
\end{multline*}
(The reader who diligently followed us when we said we do not (really) need the comultiplication, will now have to check that this assignment really defines a $*$-representation of $\pol[3]$; the reader who accepts that we have a comultiplication, may use the fact that the convolution of $*$-representations is a $*$-representation.) Now, $\rho(\U-\alpha^*)$ is injective, so $\pi(\U-u_{33})=\id\otimes\rho(\U-\alpha^*)$ is injective, too. By Relation \eqref{eq_eta_jk} for $j=1=k$, taking also into account that $\pi(u_{31})=0$, we obtain
\begin{equation} \label{n3eq}
(\id\otimes\rho(\alpha^*-\U))\eta(u_{11})
~=~
(\rho(\alpha-\U)\otimes\id)\eta(u_{33})
\end{equation}
for every cocycle with respect to $\pi$.

Suppose there was a cocycle $\eta$ with $\eta(u_{33})=e_0\otimes e_0$, so that $(\rho(\alpha-\U)\otimes\id)\eta(u_{33})=-e_0\otimes e_0$. Inserting this into \eqref{n3eq} and applying to the whole thing the map $e_0^*\otimes\id\colon x\otimes y \mapsto \langle e_0,x\rangle y$, we obtain 
$$
\rho(\alpha^*-\U)\eta(u_{11})
~=~
-e_0.
$$
However, examining what this means for the coefficients of the vector $\eta(u_{11})$, taking also into account that the products $(1-q^2)\ldots(1-q^{2k})$ converge to a non-zero limit (see \cite[Theorem A.4]{Ske94}), we would obtain $\|\eta(u_{11})\|=\infty$. Therefore, there is no such cocycle $\eta$.\end{proof}

This leaves us with the question, which vectors in $H$ may, actually, occur as values $\eta(u_{NN})$ for a cocycle. First of all, there are many of them. More precisely, every element $f$ in the (dense!) subspace $\pi(\U-u_{NN})H$ of $H$ may occur; and the (unique!) cocycles determined by them, are (exactly!) the coboundaries. Indeed, for $g=-\pi(\U-u_{NN})f$, by Lemma \ref{(1-a)lemma}, we have (see Proposition \ref{cocaprop} and around Lemma \ref{cocaGFlem} for notation)
\begin{eqnarray}
\lim_{p\uparrow1}\pi\circ(\id-\U\e)\bigl[{\textstyle\frac{-1}{\pi(\U-pu_{NN})}}g\bigr]
~=~
\pi\circ(\id-\U\e)f
~=~
(\pi f)\circ(\id-\U\e)
~=:~
\eta_f,
\end{eqnarray}
pointwise on $\pol$; and $\eta_f(u_{NN})=\pi(u_{NN}-\U)f=g$.

In general, Proposition \ref{cocaprop} tells us that, given an arbitrary cocycle $\eta$, putting
$$
f_p~:=~\frac{-1}{\pi(\U-pu_{NN})}\eta(u_{NN}),
$$
the cocycle $\eta$ is the pointwise limit of the coboundaries $\eta_{f_p}$. In other words, for each cocycle there is a sequence 
{$(f_m)_{m\in\mathbb{N}}$}
of elements in $H$ such that the coboundaries
{$(\eta_{f_m})_{m\in\mathbb{N}}$}
converge pointwise to $\eta$. Let us characterize
better, what are the sequences
{$(f_m)_{m\in\mathbb{N}}$}
that make that happen.

As in the proof of Proposition \ref{cocaprop}, it is, clearly, enough to check convergence of
{$(\eta_{f_m})_{m\in\mathbb{N}}$}
on the generators $u_{jk}$. However, we can do better.

\begin{prop}
The sequence of coboundaries
{$(\eta_{f_m})_{m\in\mathbb{N}}$}
converges (pointwise on $\pol$) if (and, of course, only if) it converges on all $u_{jj}$ $(1\le j\le N)$.
\end{prop}

\begin{proof}
We have to show that, under the stated condition,
{$\big(\eta_{f_m}(u_{jk})\big)_{m\in\mathbb{N}}$}
converges for all $j\ne k$.

Put $l:=\max(j,k)$. Then, by Relation \eqref{eq_i<k} or by Relation \eqref{eq_j<l}, we get
\begin{multline*}
\eta_{f_m}(u_{jk})
~=~
\pi(u_{jk})f_m
~=~
\frac{\pi(\U-qu_{ll})}{\pi(\U-qu_{ll})}\pi(u_{jk})f_m
\\
~=~
\frac{-1}{\pi(\U-qu_{ll})}\pi(u_{jk})\pi(u_{ll}-\U)f_m
~=~
\frac{-1}{\pi(\U-qu_{ll})}\pi(u_{jk})\eta_{f_m}(u_{ll}),
\end{multline*}
which converges, because
{$\big(\eta_{f_m}(u_{ll})\big)_{m\in\mathbb{N}}$}
converges.
\end{proof}

By introducing an adequate norm $\|\bullet\|_\pi$ on $H$, we may characterize the suitable sequences $f_m$ as Cauchy sequences in that norm and the elements in the completion $H^\pi$ with respect to that norm uniquely parametrize the cocycles.

Recall that for $0\le a\in B(H)$, the function $\|\bullet\|_a\colon f\mapsto\sqrt{\langle f,af\rangle}=\|\sqrt{a}f\|$ is a seminorm on $H$; it is a norm if and only if $a$ is injective. Choosing $a:=\sum_{j=1}^N\pi(\U-u_{jj})^*\pi(\U-u_{jj})$ to define $\|\bullet\|_\pi:=\|\bullet\|_a$, settles our problem. Indeed, because $a\ge\pi(\U-u_{NN})^*\pi(\U-u_{NN})$ and $\pi(\U-u_{NN})$ is injective, $\|\bullet\|_\pi$ is a norm (and not only a seminorm). By construction, a Cauchy sequence
{$(f_m)_{m\in\mathbb{N}}$}
in that norm leads to a pointwise convergent sequence of coboundaries
{$(\eta_{f_m})_{m\in\mathbb{N}}$.}
And since for every cocycle $\eta$ the coboundaries $\eta_{f_p}$ approximate it, the sequence
{$(f_{1-\frac{1}{m}})_{m\in\mathbb{N}}$}
is Cauchy in $\|\bullet\|_\pi$ and does the same job. We collect:

\begin{prop} \label{pinormprop}
Let $\pi$ be a representation of $\pol$ on $H$ such that $\pi(\U-u_{NN})$ is injective. Denote by $H^\pi$ the completion of $H$ in the the norm $\|\bullet\|_\pi$. Then:

For each $a\in K_1$, the operator $\pi(a)$ extends continuously to a (unique, bounded) operator
{$a^\pi\colon H^\pi\rightarrow H,f\mapsto\lim_{m\to\infty}af_m$ for an arbitrary Cauchy sequence $(f_m)_{m\in\mathbb{N}}$ converging to $f$ in $H^\pi$}.
Moreover,  the formula $a^\pi f=\eta(a)$ establishes a one-to-one correspondence between elements $f$ in $H^\pi$ and cocycles $\eta$ with respect to $\pi$.

For each $a\in K_2$ and each $f\in H^\pi$, the element $(a^\pi f)^*$ in $H^*$ extends continuously to a (unique, bounded) linear functional $g\mapsto\lim_{m\to\infty}\langle a^\pi f,g_m\rangle$ on $H^\pi$
{ for an arbitrary Cauchy sequence $(g_m)_{m\in\mathbb{N}}$ converging to $g$ in $H^\pi$.}
Denoting by $a^{\pi^*}f$ the unique element of $H^\pi$ inducing that linear functional, we define an operator $a^{\pi^*}\colon H^\pi\rightarrow H^\pi$, fulfilling also $\langle g,a^{\pi^*} f\rangle_\pi=\lim_{m\to\infty}\langle g_m,\pi(a) f_m\rangle$. In particular, if the cocycle $\eta$ 
{ is}
determined by $f\in H^\pi$, then the linear functional $a\mapsto\langle f,a^{\pi^*} f\rangle_\pi$ on $K_2$
{ may be}
extended to a generating functional $\psi$ with Sch{\"u}rmann triple $(\pi,\eta,\psi)$.
\end{prop}

{This result generalizes the
situation of $SU_q(2)$ as described in \cite[Section 4.5]{Ske94} or \cite[Section 4.3]{skeide99}.}

\begin{cor} \label{cngpincor}
Let $\pi$ be a completely non-gaussian representation (that is, in the decomposition according to Theorem \ref{psidecthm}, $\pi_1=0$). Then $\|\bullet\|_\pi$, defined as before, is a norm and, defining $H^\pi$ as before, Proposition \ref{pinormprop} remains true also for $\pi$.
\end{cor}

\begin{proof}
For each nonzero $h=h_2+\ldots+h_N\in H$, there is at least one $n$ such that $h_n\ne0$. And since $\pi_n(\U-u_{nn})$ is injective on $H_n$, the seminorm $\|\bullet\|_\pi$ is, indeed, a norm. Furthermore, since $\pi_n=\tilde{\pi}_n\circ s_{n,N}$ lives on $\pol[n]$, we have  $\|\bullet\|_{\pi_n}=\|\bullet\|_{\tilde{\pi}_n}$. Now, the result follows by applying Proposition \ref{pinormprop} to each $\tilde{\pi}_n$, separately.\end{proof}

Putting this together with Theorem \ref{psidecthm},  with Theorem \ref{Gclassthm}, and with the uniqueness discussion for Sch{\"u}rmann triples following Definition \ref{def_triple}, we obtain the following improvement of the parametrization following Procedure \ref{proc}:

\begin{thm} \label{mainthm}
We obtain every generating functional $\psi$ on $\pol$ in the following way as
$$
\psi
~=~
\psi_G+\psi_L,
$$
where:
\begin{enumerate}
\item
$\psi_G$ is gaussian.

\item
$\psi_L$ is completely non-gaussian.
\end{enumerate}
In the decomposition of a given $\psi$ as $\psi_G+\psi_L$, both $\psi_G$ and $\psi_L$ are maximal. They are unique under the condition $\psi_L\circ\cP=\psi_L$.
\begin{enumerate}
\item
The possible gaussian parts are classified uniquely by positive real $(N-1)\times(N-1)$-matrices (encoding uniquely the part $\psi_G\circ\cP$ of $\psi_G$) and $N-1$ real numbers (encoding the remaining drift term).

\item
The possible completely non-gaussing parts (satisfying $\,\psi_L\circ\cP=\psi_L$) are classified (uniquely up to cyclic cocycle intertwining unitary equivalence) by completely non-gaus\-sian $*$-representations $\pi$ on a Hilbert space and elements $f\in H^\pi$ (such that the cocycle determined by $f$ is cyclic) as
$$
\psi_L(a)
~=~
\langle f,\cP(a)^{\pi^*}f\rangle_\pi.
$$
\end{enumerate}
If we wish, we may decompose $\pi$ and $H$ uniquely further as in Theorem \ref{repdecthm} into $*$-representations $\pi_n$ on $H_n$ $(2\le n\le N)$ such that $\pi(\U-u_{nn})$ is injective and $f_n\in H_n^{\pi_n}$ with
$$
\psi_L
~=~
\psi_2+\ldots+\psi_N,
$$
where $\psi_n=\langle f_n,\cP(a)^{\pi_n^*}f_n\rangle$ (preserving analogue uniqueness statements).
\end{thm}

\newpage

\addtocontents{toc}{\vspace{1.5ex}}
\section{The case of $U_q(N)$} \label{sec_uqn}
\noindent
Recall the definition of $U_q(N)$ from Subsection \ref{ssec_suqn}. And recall that we have the inclusions
$$
SU_q(N+1)
~\supset~
U_q(N)
~\supset~
SU_q(N)
$$
mediated by the $*$-homomorphisms $t_N\colon\pol[N+1]\rightarrow U_q(N)$ and $\breve{t}_{N-1}\colon U_q(N)\rightarrow\pol$, respectively. In this short section we show the Lévy-Khintchine decomposition results for generating functionals on $U_q(N)$, by reduction to those for $SU_q(N+1)\supset U_q(N)$. This is very rapid. The price to be paid for being rapid, is that, by this method, we do not get the full decomposition result for the completely non-gaussian part into irreductible parts the live on $U_q(n)$, but only the weaker Proposition \ref{Undecprop}.

\lf
\begin{remark}
One may repeat the whole procedure in Section \ref{decompSEC} almost verbatim for $U_q(N)$ to show that every generating functional $\psi$ on $U_q(N)$ decomposes (in a suitable sense uniquely) into a sum $\psi=\psi_0+\ldots+\psi_N$, where $\psi_n$ lives on $U_q(n)$ for $n=1,\ldots,N$ and where $\psi_0$ is the gaussian part. (Some care is in place at points where in Section \ref{decompSEC} we did use the determinant condition $D=\U$ that characterizes $\pol$.) We opted not to include details.
\end{remark}

Since $SU_q(N+1)\supset U_q(N)\supset\mathbb{T}^N$, by Remark \ref{TNrem}, $U_q(N)$ has the same gaussian generating functionals as $SU_q(N+1)$ and the images $t_N(d_2)=d_2,\ldots,t_N(d_N)=d_N,t_N(d_{N+1})=\frac{D^{-1}-{D^{-1}}^*}{2i}$ form a suitable family $E_1$ for $U_q(N)$, defining also a projection $\cP$ compatible with that of $SU_q(N+1)$.

Now if $\pi$ is a $*$-representation of $U_q(N-1)$, then $\hat{\pi}:=\pi\circ t_{N-1}$ is a $*$-representation of $\pol$ that lives on $U_q(N-1)$. Obviously, if a representation of $\pol$ that lives on $U_q(N-1)$ decomposes into a direct sum, then each direct summand lives on $U_q(N-1)$, separately. Therefore, all results about representations of $\pol$ in Subsection \ref{repdSSEC}, turn over to $U_q(N-1)$ (including the classical $U_q(1)=U(1)$ for $N=2$).

If $\hat{\eta}$ is a cocycle with respect to that $\hat{\pi}$, then, by Subsection \ref{cocaSSEC}, the completely non-gaussian part $\hat{\eta}_{NG}:=\hat{\eta}_2\oplus\ldots\oplus\hat{\eta}_N$, is a limit of coboundaries with respect to a representation $\hat{\pi}_{NG}:=\hat{\pi}_2\oplus\ldots\oplus\hat{\pi}_N$ that lives on $U_q(N-1)$. (Indeed, if a direct sum of representations lives on a quantum subgroup, then so does each its components.) By Corollaries \ref{subapprcor} and \ref{subPcor}, also $\hat{\eta}_{NG}$ lives on $U_q(N-1)$ and admits a generating functional $\hat{\psi}_{NG}$ (unique, if $\hat{\psi}_{NG}\circ\cP=\hat{\psi}_{NG}$) completing the Sch{\"u}rmann triple $(\hat{\pi}_{NG},\hat{\eta}_{NG},\hat{\psi}_{NG})$ that lives on $U_q(N-1)$, too. Therefore, if $\hat{\eta}$ lives on $U_q(N-1)$, then so does the gaussian part $\hat{\eta}_G:=\hat{\eta}_1$. And if $(\hat{\pi},\hat{\eta},\hat{\psi})$ is obtained from a Sch{\"u}rmann triple $(\pi,\eta,\psi)$ on $U_q(N-1)$ by composition with $t_{N-1}$ (and, therefore, is a Sch{\"u}rmann triple on $\pol$), then the gaussian part $\hat{\psi}_G:=\hat{\psi}-\hat{\psi}_{NG}$ lives on $U_q(N-1)$, too.

Except for some care about $N\leftrightarrow N-1$, taking also into account the analogue of Corollary \ref{cngpincor}, we obtain word by word the analogue of  the first part of Theorem \ref{mainthm}.

\begin{thm} \label{Umainthm}
For $N\ge 1$, we obtain every generating functional $\psi$ on $U_q(N)$ in the following way as
$$
\psi
~=~
\psi_G+\psi_L,
$$
where:
\begin{enumerate}
\item
$\psi_G$ is gaussian.

\item
$\psi_L$ is completely non-gaussian.
\end{enumerate}
In the decomposition of a given $\psi$ as $\psi_G+\psi_L$, both $\psi_G$ and $\psi_L$ are maximal. They are unique under the condition $\psi_L\circ\cP=\psi_L$.
\begin{enumerate}
\item
The possible gaussian parts are classified uniquely by positive real $N\times N$-matrices (encoding uniquely the part $\psi_G\circ\cP$ of $\psi_G$) and $N$ real numbers (encoding the remaining drift term).

\item
The possible completely non-gaussing parts (satisfying $\,\psi_L\circ\cP=\psi_L$) are classified (uniquely up to cyclic cocycle intertwining unitary equivalence) by completely non-gaus\-sian $*$-representations $\pi$ on a Hilbert space and elements $f\in H^\pi$ (such that the cocycle determined by $f$ is cyclic) as
$$
\psi_L(a)
~=~
\langle f,\cP(a)^{\pi^*}f\rangle_\pi.
$$
\end{enumerate}
\end{thm}

\begin{cor}
$U_q(N)$ does have property {\normalfont(NC)}, hence, {\normalfont(LK)}. It has property {\normalfont(GC)} if and only if  $N=1$.
\end{cor}

What about the decomposition of the non-gaussian part into components living on $U_q(n)$? (This is the part from Theorem \ref{mainthm} that is missing in Theorem \ref{Umainthm}.) Well, returning to the notation in front of Theorem \ref{Umainthm} but now immediately for $U_q(N)$ and no longer for $U_q(N-1)$, we know that $\hat{\psi}_{NG}$ decomposes into a sum over $\hat{\psi}_{n+1}$ $(1\le n\le N)$, where $\hat{\psi}_{n+1}$ lives on $SU_q(n+1)$ and on $U_q(N)$. But, does it live on $U_q(n)\subset SU_q(n+1)$? We have $\hat{\psi}_{n+1}=\psi_n\circ t_N$ and we have, making also use of \eqref{snN}, \eqref{BsnN}, and \eqref{sBstCR},
$$
\hat{\psi}_{n+1}
~=~
\tilde{\psi}_{n+1}\circ s_{n+1,N+1}
~=~
\tilde{\psi}_{n+1}\circ\breve{t}_{n+1}\circ\breve{s}_{n+1,N}\circ t_N.
$$
Therefore, by surjectivity of $t_N$, the functional $\psi_n$ on $U_q(N)$ fulfills
$$
\psi_n
~=~
(\tilde{\psi}_{n+1}\circ\breve{t}_{n+1})\circ\breve{s}_{n+1,N}
$$
and lives on $U_q(n+1)$ via the functional $\tilde{\psi}_{n+1}\circ\breve{t}_{n+1}$ on $U_q(n+1)$ which, in turn, lives on $SU_q(n+1)$.

Every such $\tilde{\psi}_{n+1}$ may occur; therefore, better than this is not possible as long as we work by reduction to the results for $\pol[N+1]$:

\begin{prop} \label{Undecprop}
\hspace{1ex}
The completely non-gaussian part $\psi_{NG}$ of a generating functional on $U_q(N)$ $(N\ge2)$ decomposes uniquely as $\psi_2+\ldots+\psi_{N+1}$, where each $\psi_n$ lives on $SU_q(n)$, hence, on
$U_q(n-1)$,
and where as functionals on $SU_q(n)$, the $\psi_n$ are irreductible.
\end{prop}

\newpage

\addtocontents{toc}{\vspace{1.5ex}}
\section{Final remarks and open problems} 
\label{sec_final}
We close by pointing out that several interesting open problems related to our decomposition result. 

\begin{enumerate}
\item
Theorem \ref{mainthm} parametrizes all possible generating functionals on $SU_q(N)$, but it is possible that two sets of parameters lead to the same generating functional. When can this happen? What is the fundamental domain of this equivalence relation? That is, how to choose and characterize one representative for any equivalence class? The answer to this question will establish the one-to-one correspondence between L\'evy processes and the parameters.  

\item \label{two}
We assume in the paper that $q\in (0,1)$, but the same holds for $|q|<1$, $q\neq 0$.
The (classical) limit case $q=1$ is known. It would be of great interest to see what happens for $q=-1$. The treatment of $SU_{-1}(2)$, that is, the anti-classical limit, was done in \cite{Ske99b}. 

\item
It would be interesting to generalize our results to q-deformations of other simple compact Lie groups. Descriptions of the CQG-algebras of the compact quantum groups $O_q(N)$, $Sp_q(N)$, and $SO_q(N)$ can be found in Chapter 9 of \cite{KlSchm97}.

\item
As shown in \cite{cfk14}, generating functionals on a given compact quantum group which satisfy additional symmetry properties (KMS-symmetry) can be used to define Dirichlet forms, Laplace operators and Dirac operators, and so they can carry geometric information about the quantum group. A description of all generating functionals satisfying additional properties (KMS-symmetric, central, etc.) will allow to construct ``nice'' Dirichlet forms, Laplace operators and Dirac operators that reflect well the structure of the underlying quantum group.

\item
We have decomposed  generating functionals  on $SU_q(N)$ into the sum of a gaussian part and of a completely non-gaussian part (L{\'e}vy-Khintchine decomposition), and we have decomposed the completely non-gaussian part further into into a sum of generating functionals that live on the quantum subgroups $SU_q(n)$ $(2\le n\le N)$ and, there, are irreductible. 
As soon as we actually have examples of L{\'e}vy processes on $SU_q(N)$ that have gaussian or irreductible generating functionals, we may ask how to ``compose'' them to get a L{\`e}vy processs for the sum of the generators.

Of course, this question makes sense independently of the special nature of the generating functionals: Given two (or more) L{\'e}vy processes on a quantum semigroup with generating functionals $\psi_i$, how to construct out of them a L{\'e}vy process that has the sum over the $\psi_i$ as generating functional? Answer: Plugging in the direct sum of the Sch{\"u}rmann triples of the $\psi_i$ into the construction in Sch{\"u}rmann, Skeide, and Volkwardt \cite{ScSV10}. (See \cite[Example 3.13]{ScSV10} for details.) The result is that a certain {\it convolution Trotter product} of the individual L{\'e}vy processes gives a L{\'e}vy process that has as generating functional the sum of the individual generating functionals.

\end{enumerate}


\newpage

\addtocontents{toc}{\vspace{1.5ex}}

\noindent
{\bf Acknowledgments:} We gratefully acknowledge the MFO in Oberwolfach for fantastic two weeks during a Research in Paris in February 2014, where this work started.

{U.F.\ was supported by the French `Investissements d'Avenir' program, project ISITE-BFC (contract ANR-15-IDEX-03) and by an ANR project (No. ANR-19-CE40-0002).

AK was supported by the Polish National Science Center grant SONATA 2016/21/D/ST1 /03010 and by the Polish National Agency for Academic Exchange in frame of POLONIUM program PPN/BIL/2018/1/00197/U/00021. }

\lf\noindent
Uwe Franz: {\it Laboratoire de math\'ematiques de Besan\c{c}on, University of Bourgogne Franche-Comt\'e, France}, \\
E-mail: \href{mailto:uwe.franz@univ-fcomte.fr}{\tt{uwe.franz@univ-fcomte.fr}},\\
Homepage: \url{http://lmb.univ-fcomte.fr/uwe-franz}

\lf\noindent
Anna Kula: {\it Institute of Mathematics, University of Wroclaw, Poland}, \\
E-mail: \href{mailto:anna.kula@math.uni.wroc.pl}{\tt{anna.kula@math.uni.wroc.pl}},\\
Homepage: \url{http://www.math.uni.wroc.pl/~akula/}

\lf\noindent
J.\ Martin Lindsay: {\it Department of Mathematics and Statistics, Lancaster University, UK}, \\
E-mail: \href{mailto:j.m.lindsay@lancaster.ac.uk}{\tt{j.m.lindsay@lancaster.ac.uk}}

\lf\noindent
Michael Skeide: {\it Dipartimento di Economia, Universit\`{a} degli Studi del Molise, Via de Sanctis, 86100 Campobasso, Italy}, \\
E-mail: \href{mailto:skeide@unimol.it}{\tt{skeide@unimol.it}},\\
Homepage: \url{http://web.unimol.it/skeide/}

\end{document}